\documentclass[a4paper,legno,10pt]{article}

\usepackage[utf8]{inputenc}
\usepackage[T1]{fontenc}
\usepackage{xspace,
	amsthm,
	amsmath,
	amssymb,
	bbm,
	tikz,
	graphicx,
	subfig,
	keyval}
\DeclareGraphicsExtensions{.png, .pdf}

\usepackage[hidelinks]{hyperref}

\usepackage{geometry}
\tikzstyle{nodino}=[circle,draw,fill,inner sep=0pt,minimum size=0.5mm]
\tikzstyle{infinito}=[circle,inner sep=0pt,minimum size=0mm]
\tikzstyle{nodo}=[circle,draw,fill,inner sep=0pt, minimum size=0.5*width("k")]
\tikzstyle{nodo_vuoto}=[circle,draw,inner sep=0pt, minimum size=0.5*width("k")]
\usetikzlibrary{graphs}
\tikzset{every loop/.style={min distance=10mm,in=300,out=240,looseness=10}}
\tikzset{place/.style={circle,thick,draw=blue!75,fill=blue!20,minimum
		size=6mm}}
\tikzset{place2/.style={circle,thick,draw=red!75,fill=red!20,minimum
		size=6mm}}

\newcommand{\rr}{{\mathbb R}}

\newcommand{\zz}{{\mathbb Z}}
\newcommand{\nn}{{\mathbb N}}
\newcommand{\G}{{\mathcal{G}}}
\newcommand{\udot}{\|u'\|_{L^2(\mathcal{G})}}
\newcommand{\uLp}{\|u\|_{L^p(\mathcal{G})}}
\newcommand{\uLtwo}{\|u\|_{L^2(\mathcal{G})}}
\newcommand{\HmuG}{H_\mu^1(\mathcal{G})}
\newcommand{\uLsix}{\|u\|_{L^6(\mathcal{G})}}

\newcommand{\K}{\mathcal{K}}

\newcommand{\dx}{\,dx}

\theoremstyle{plain} 
\newtheorem{thm}{Theorem}[section]

\newtheorem{prop}[thm]{Proposition} 

\theoremstyle{definition}

\theoremstyle{definition} 
\newtheorem{defn}{Definition}[section]

\theoremstyle{remark} 
\newtheorem{rem}{Remark}[section]

\title{Mass-constrained ground states of the stationary NLSE on periodic metric graphs}

\author{Simone Dovetta
	\\ \ \\ \ \\
	{\small  Dipartimento di Scienze Matematiche ``G.L. Lagrange'', Politecnico di Torino} \\
	{\small Corso Duca degli Abruzzi, 24, 10129 Torino, Italy} \\ 
	{\small and} \\
	{\small Dipartimento di Matematica ``G. Peano'', Universit\`a degli Studi di Torino }\\
	{\small Via Carlo Alberto, 10, 10123, Torino, Italy} \\
	{\small \texttt{simone.dovetta@polito.it}}
}

\begin{document}
	
	\maketitle
	
	\begin{abstract}
	We investigate the existence of ground states with fixed mass for the nonlinear Schrödinger
	equation with a pure power nonlinearity on periodic metric graphs. Within a variational
	framework, both the $L^2$-subcritical and critical regimes are studied. In the former case,
	we establish the existence of global minimizers of the NLS energy for every mass and every
	periodic graph. In the critical regime, a complete topological characterization
	is derived, providing conditions which allow or prevent ground states of a certain mass from existing. Besides, a rigorous notion of periodic graph is introduced and discussed.	
	\end{abstract}
	
	\section{Introduction}
	
	 Originally fuelled by a wide variety of physical applications, the theory of quantum graphs has
	 nowadays become a prominent topic of research. Moving from linear dynamics on branched
	 structures (see for instance \cite{exner,friedlander} and the monograph \cite{berkolaiko_kuchment}), nonlinear problems have been
	 extensively studied first on star-graphs \cite{squadra_azzurra1,squadra_azzurra2,squadra_azzurra3,squadra_azzurra4}, and more recently on general non-compact
	 metric graphs with at least one half-line \cite{AST calc var,AST funct an,AST critico,AST bound states,bct,dt_p,dt,serra_tentarelli1,serra_tentarelli2,tentarelli}, as well as on compact graphs
	 \cite{cds,dovetta}.
	
	
	\begin{figure}[t]
		\centering
		\subfloat[][]{
			\includegraphics[width=0.35\columnwidth]{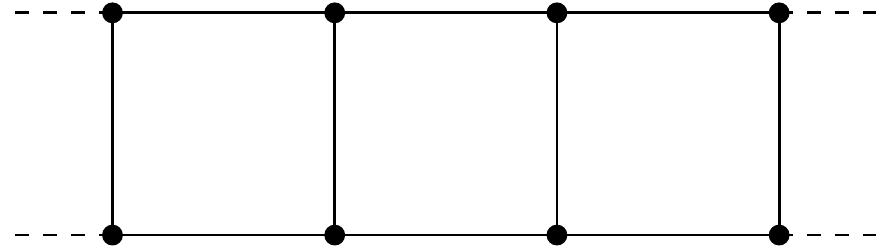}}\qquad
		\subfloat[][]{
		\includegraphics[width=0.5\columnwidth]{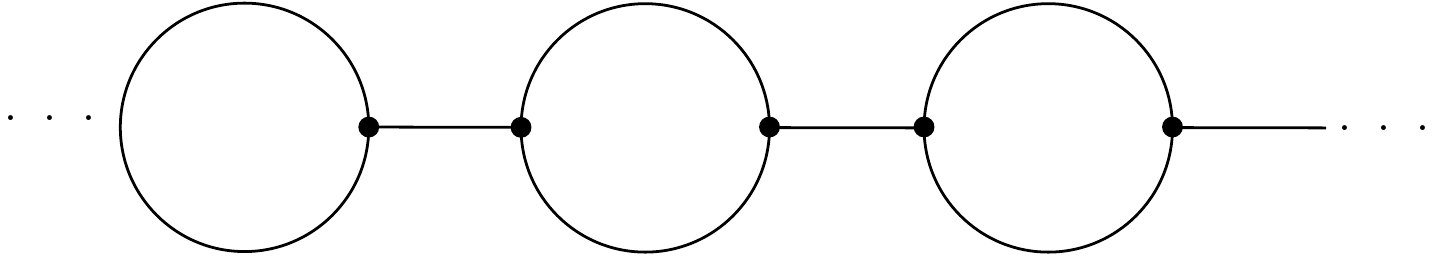}}
	    \caption{Examples of periodic graphs: a ladder graph (a) and a graph made up by circles and segments (b).}
	    \label{FIG-ladder}
	\end{figure}

	
	Among the whole theory, periodic graphs appear to gather a significant interest. As linear problems and thoroughly spectral analysis has been carried on for instance in \cite{exner0,exner1,exner2}, a first investigation of NLS equation on ladder-type graphs (Figure \ref{FIG-ladder}(a)) was initiated in \cite{nakamura}, while spectral problems for the graph in Figure \ref{FIG-ladder}(b) were discussed both in \cite{pely} and in \cite{pely_schneider}. Recently, a variational exploration of the NLS energy on general periodic graphs has been developed in \cite{pankov}, where a generalized Nehari manifold approach is used to establish for the first time the existence of a global minimizer.
	
	In this paper we investigate the existence of \emph{ground states} of the NLS energy functional
	
	\begin{equation}
		\label{EQ-def energy intro}
		E(u,\G)=\frac{1}{2}\udot^2-\frac{1}{p}\uLp^p=\frac{1}{2}\int_\G|u'|^2\,dx-\frac{1}{p}\int_\G|u|^p\,dx
	\end{equation}
	
	\noindent on a general \emph{periodic metric graph} $\G$, with $p\in(2,6]$, under the \emph{mass constraint}
	
	\begin{equation}
		\label{EQ-def mass constraint}
		\uLtwo^2=\mu>0\,.
	\end{equation}
	
	\noindent In what follows, we restrict ourselves to consider only real-valued functions.
	
	The class of graphs we consider here is rather general. Roughly speaking, we say that
	a graph $\G$ is periodic if it is built of an infinite number of copies of a fixed compact graph, the \textit{periodicity cell}, glued together along one direction, i.e., we deal with structures enjoying a $\zz$-symmetry (see section \ref{sec:periodic}).
	
	We briefly recall that a connected metric graph $\G = (V(\G),E(\G))$ is a connected space made up
	by segments of line, the \textit{edges}, glued together at some points, the \textit{vertices}, according to the topology of the graph. Both multiple edges and self-loops are possible. On every edge $e\in E(\G)$,
	a coordinate $x_e$ is defined, providing an identification of $e$ with a real interval $I_e=[0,\ell_e]$.
	For our purposes here, we always have $\ell_e<+\infty$, i.e., all edges in the graph are bounded.
	Moreover, without loss of generality, we assume that all the periodic graphs we consider possess at least one vertex of degree at least 3, in order to exclude the situation $\G=\rr$.
	
	Within this framework, a function $u:\G\to\rr$ can be seen as a family $\{u_e\}_{e\in E(\G)}$, where $u_e:I_e\to\rr$ defines the restriction of $u$ to the edge $e$, and functional spaces can be defined in the natural way
	
	\[
	\begin{split}
	L^p(\G):=&\{u:\G\to\rr\,:\,u_e\in L^p(I_e),\,\forall e\in E(\G)\}\\
	H^1(\G):=&\{u:\G\to\rr\,\,\text{continuous}\, :\, u_e\in H^1(I_e),\,\forall e\in E(\G)\}\,.
	\end{split}
	\]
	
	\noindent Note that, since all $u_e$ are one-dimensional, the continuity condition we introduce in the
	definition of $H^1(\G)$ is meant to impose $u$ to be continuous at the vertices.
	Moreover, as we are looking for functions satisfying \eqref{EQ-def mass constraint}, it is useful to introduce, for $\mu>0$, the mass-constrained space
	
	\[
	H_\mu^1(\G):=\{u\in H^1(\G)\,:\,\uLtwo^2=\mu\}\,.
	\]
	
	\noindent By ground states we mean global minimizers of the energy \eqref{EQ-def energy intro}, that are solutions, for a suitable Lagrange multiplier $\lambda$, of the \textit{stationary Schr\"odinger equation} with focusing nonlinearity
	
	\begin{equation}
		\label{EQ-NLSE}
		u''+|u|^{p-2}u=\lambda u
	\end{equation}
	
	\noindent on each edge of $\G$, with homogeneous \textit{Kirchhoff conditions} at every vertex, that is, the oriented sum of all derivatives entering the vertex is equal to zero
	
	\[
	\sum_{e\succ v}\frac{du}{dx_e}(v)=0
	\]
	
	\noindent (see \cite{AST funct an,AST parma} for a discussion on this condition).
	
	Our approach to the problem is variational. However, the strategy we follow here is significantly different from the one in \cite{pankov}, as extending similar arguments to the mass-constrained setting is far from obvious.
	
	In the $L^2$-subcritical regime $p\in(2,6)$, we prove existence of ground states, always realizing strictly negative energy, for every value of the parameter $\mu$. Indeed, denoting by
	
	\[
	\mathcal{E}_\G(\mu):=\inf_{u\in\HmuG}E(u,\G)
	\]
	
	\noindent the \textit{ground state energy level} of $E$, our first result is the following.
	
	\begin{thm}
		\label{THM 1 INTRO}
		Let $\G$ be a periodic graph and $p\in(2,6)$. Then, for every $\mu>0$
		
		\begin{equation}
			\label{EQ-inf negative THM 1}
			-\infty<\mathcal{E}_\G(\mu)<0
		\end{equation}
		
		\noindent and there always exists a ground state with mass $\mu$, i.e., $u\in\HmuG$ such that $\mathcal{E}_\G(\mu)=E(u,\G)$.
	\end{thm}

	Theorem \ref{THM 1 INTRO} unveils a similarity with the real line $\rr$, for which it is known that ground states of the energy are the unique (up to symmetries) solutions of \eqref{EQ-NLSE} with prescribed mass (see for
	instance \cite{cazenave} and Section \ref{sec:subcritical} below).
	
	Actually, it is at the critical exponent $p = 6$ that the problem exhibits a wider variety of behaviours, as the topology of the graph enters the game. Recall that (see \cite{cazenave} and Section \ref{sec:critical} here), when $p = 6$, both for the real line $\rr$ and half-line $\rr^+$, ground states exist if and only if the mass is equal to a threshold value, denoted by $\mu_\rr,\,\mu_{\rr^+}$, respectively. 
	
	For a general periodic graph $\G$, we show that, a critical mass $\mu_\G$ naturally arises as well, its actual value being determined by the specific structure of the graph (see Section \ref{sec:critical}). Moreover, the relation between this threshold and the existence of ground states is more complex than in the case of $\rr$ and $\rr^+$, and the situation changes with respect to the graph we are dealing with.
	
	A key-role in this context is played by the following topological condition, denoted by (H$_{per}$)
	
	\[
	\text{(H}_{per})\,:\,\text{removing any edge of }e\in E(\G)\text{ generates only non-compact connected components.}
	\]
	
	\noindent Assumptions of this fashion have been introduced for the first time in \cite{AST calc var} for graphs with half-lines, named assumption (H), of which (H$_{per}$) here constitutes the periodic version (for a detailed overview on equivalent formulations of (H) we refer to \cite{AST parma}). The key idea behind this condition is that, if $\G$ satisfies (H$_{per}$) , then, for every point $x\in\G$, two disjoint paths of infinite
	length originating at $x$ exist.
	
	We then state our theorems. Recall that a \textit{terminal edge} denotes an edge incident to a vertex of degree 1.
	
	
	\begin{figure}[t]
		\centering
		\subfloat[][]{
			\includegraphics[width=0.5\columnwidth]{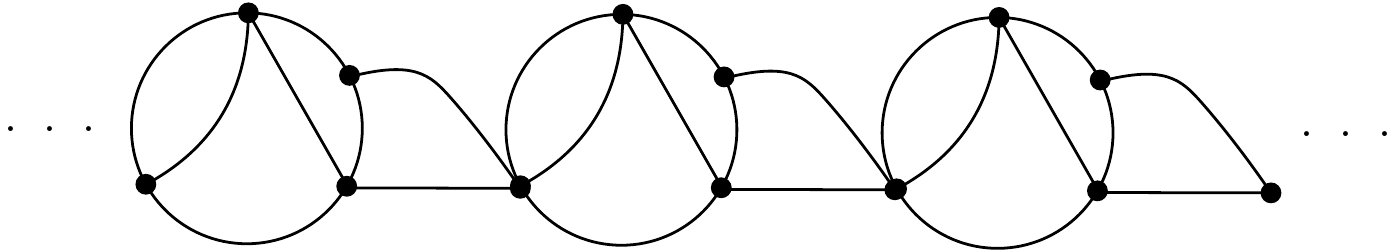}} \qquad
		\subfloat[][]{
			\includegraphics[width=0.4\columnwidth]{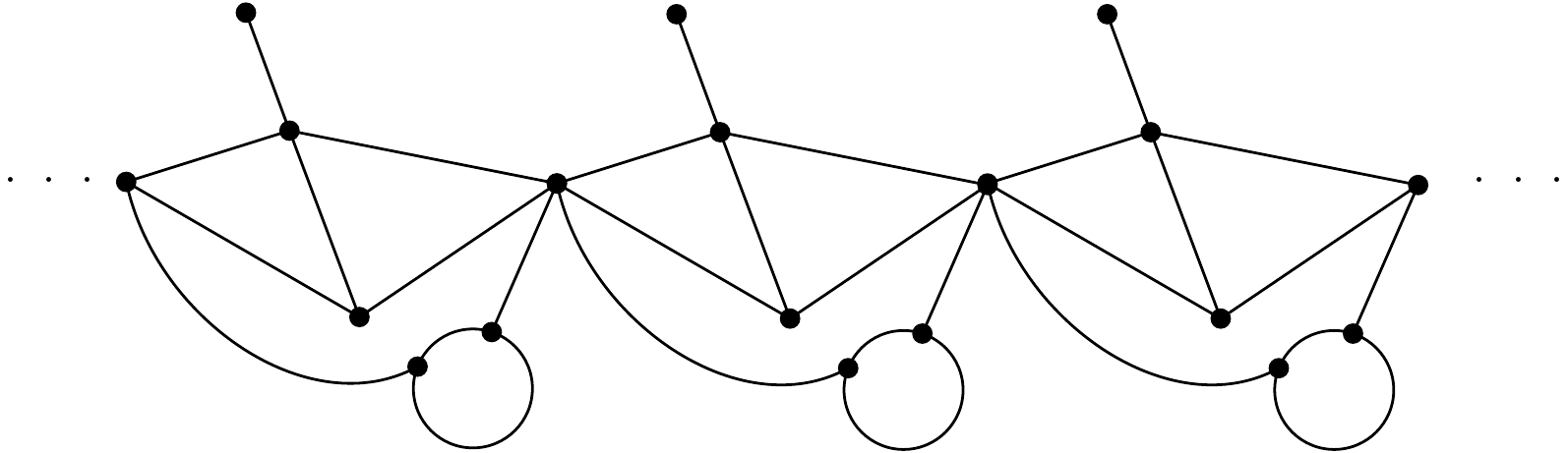}}\\
		
		\subfloat[][]{
			\includegraphics[width=0.5\columnwidth]{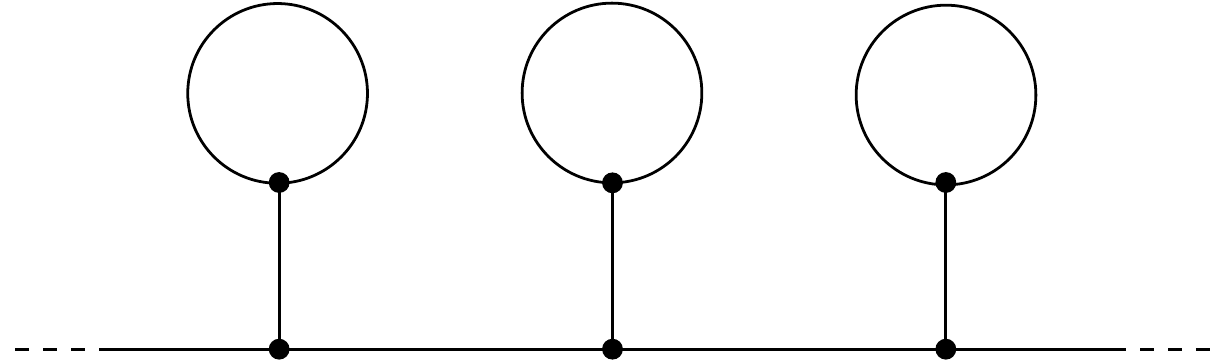}}
		\caption{examples of periodic graphs satisfying assumption (H$_{per}$) (a), with a terminal edge (b) and violating (H$_{per}$) without a terminal edge (c).}
		\label{FIG-critical}
	\end{figure}
	
	
	\begin{thm}
		\label{THM 2 INTRO}
		Let $\G$ be a periodic graph and $p=6$. Then:
		\begin{itemize}
			\item[(i)] if $\G$ satisfies assumption (H$_{per}$) (Figure \ref{FIG-critical}(a)), then $\mu_\G=\mu_\rr$ and
			\begin{equation}
				\label{EQ-inf p=6 H per}
				\mathcal{E}_\G(\mu)=\begin{cases}
				0 & \text{if }\mu\leq\mu_\rr\\
				-\infty & \text{if }\mu>\mu_\rr\,;
				\end{cases}
			\end{equation}
			\item[(ii)] if $\G$ has a terminal edge (Figure \ref{FIG-critical}(b)), then $\mu_\G=\mu_{\rr^+}$ and
			\begin{equation}
				\label{EQ-inf p=6 pendant}
				\mathcal{E}_\G(\mu)=\begin{cases}
				0 & \text{if }\mu\leq\mu_{\rr^+}\\
				-\infty & \text{if }\mu>\mu_{\rr^+}\,.
				\end{cases}
			\end{equation} 
		\end{itemize}
	\noindent Moreover, the infimum is never attained.
	\end{thm}

	\begin{thm}
		\label{THM 3 INTRO}
		Let $\G$ be a periodic graph violating assumption (H$_{per}$) and with no terminal edge (Figure \ref{FIG-critical}(c)), and $p=6$. If $\mu_\G<\mu_\rr$, then ground states with mass $\mu$ exist if and only if $\mu\in[\mu_\G,\mu_\rr]$.
	\end{thm}
	
	Theorems \ref{THM 2 INTRO}-\ref{THM 3 INTRO} provide a complete topological characterization of the existence of ground states in the critical setting (similarly to what reported in \cite{AST critico} for graphs with half-lines). 
	
	On the one hand, it turns out that graphs satisfying (H$_{per}$) behave almost as $\rr$, whereas the ones with a terminal edge fake the half-line $\rr^+$, as the values of the corresponding thresholds are respectively the same. However, for both these classes of graphs ground states never exist, even at the critical masses. Thus, both (H$_{per}$) and the presence of a terminal edge provide topological sufficient conditions preventing the existence of global minimizers. 
	
	On the other hand, for all other graphs, global minimizers actually exist for a whole interval of masses provided $\mu_\G<\mu_\rr$, and we also show that the class of graphs fulfilling this assumption is nonempty (see Proposition \ref{prop-signpost}). However, it is still unclear whether such a condition is immediately satisfied by violating (H$_{per}$), so up to now it is necessary to impose it to recover the above existence result.
	
	To conclude this Introduction, we wish to stress once more the fact that all our results hold
	for periodic graphs in which each periodicity cell shares connections with exactly two of the others, so that the whole graph displays a $\zz-$symmetry. If such a condition is removed, and one allows for repetitions of the periodicity cell along more than one direction, the situation drastically changes and the possible behaviours seem to be sensitively varying (see Figure \ref{FIG-grid} for some examples). Further investigations in this direction have been recently initiated in \cite{ADST} for the so-called doubly periodic graphs (i.e., graphs with a $\zz^2$-symmetry) as the two-dimensional grid in Figure \ref{FIG-grid}(b), where threshold phenomena have been shown to appear not only at $p=6$ but for a whole interval of exponents. 
	
	\medskip
	The paper is organised as follows. In Section \ref{sec:periodic} we discuss the definition of periodicity, characterizing the class of graphs we are considering. Section \ref{sec:subcritical} deals with the subcritical regime, whereas Section \ref{sec:critical} is devoted to the critical case. Finally, Appendix \ref{appendix} runs through the generality of the definition of periodicity given in Section \ref{sec:periodic} from a graph theoretical point of view. 
	
	
	\begin{figure}[t]
		\centering
		\subfloat[][]{
		\includegraphics[width=0.4\textwidth]{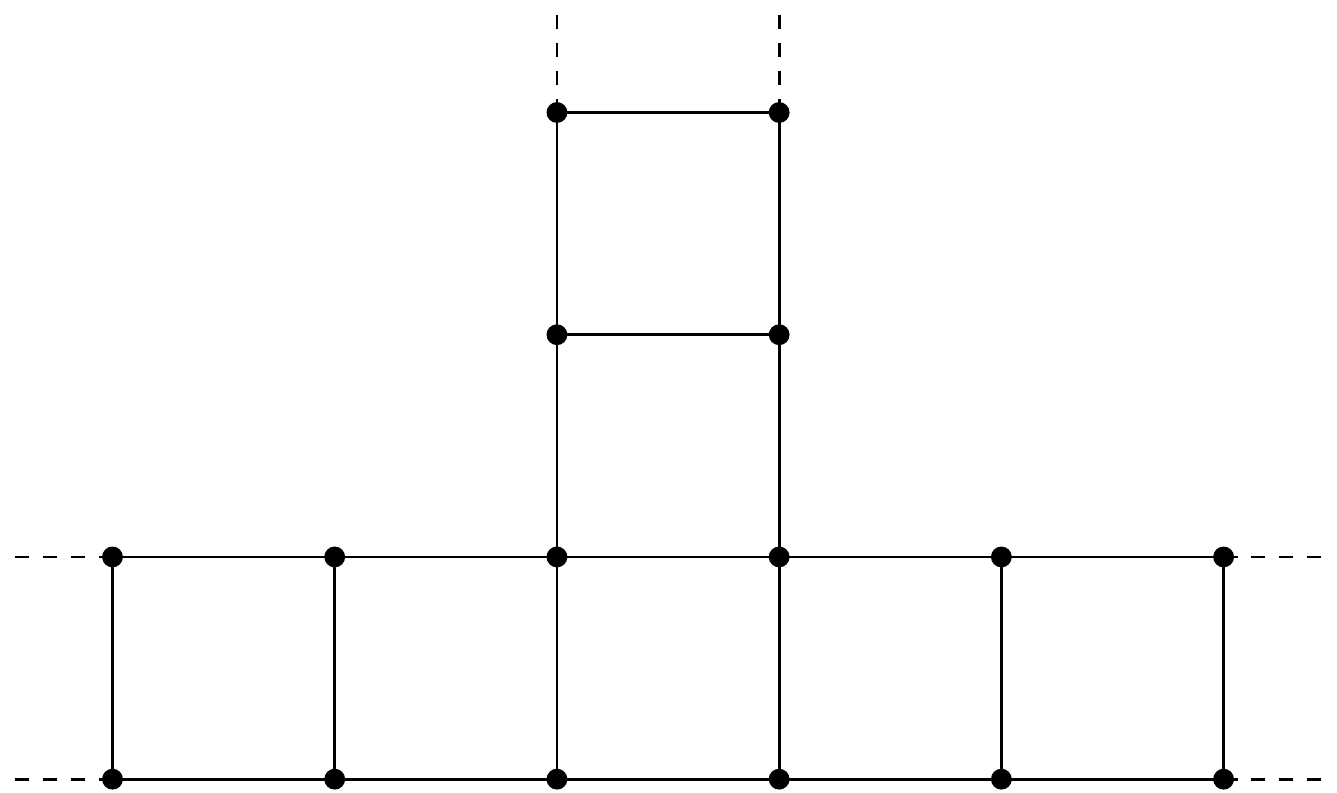}}\qquad
		\subfloat[][]{
		\includegraphics[width=0.25\textwidth]{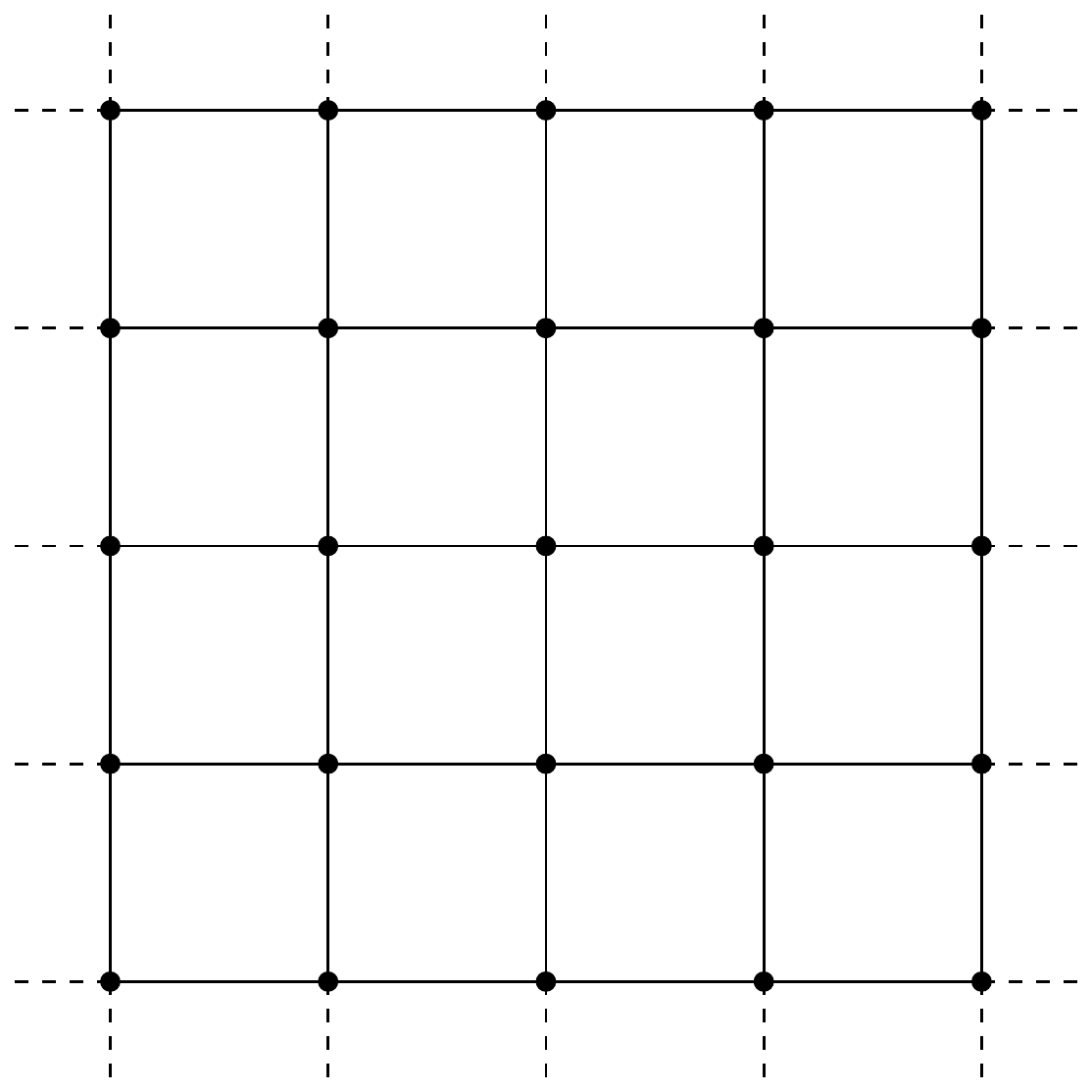}}
		\caption{examples of graphs in which the periodicity cell is repeated in two directions (a) and in an infinite number of directions (b).}
		\label{FIG-grid}
	\end{figure}


	\section{Periodic graphs: formal definition}
	\label{sec:periodic}
	
	The aim of this section is to provide a rigorous definition of what we mean by \textit{periodic graphs}.
	
	To this purpose, let us begin by recalling the approach of \cite{pankov}, where the periodicity of a graph is described in terms of a proper group action (see also Chapter 4 in \cite{berkolaiko_kuchment}). Indeed, let $\G$ be a connected metric graph with sets of vertices and edges $V(\G)$, $E(\G)$, respectively, and consider an action of the group $\zz^n$ on $\G$
	
	\[
	\zz^n\times\G\ni(g,x)\mapsto gx\in\G
	\]
	
	\noindent which is a graph automorphism, i.e., it maps vertices into vertices and edges into edges, and it preserves the lengths, i.e., for every interval $I\subset e\in E(\G)$ and every $g\in\zz^n$, both $I$ and $gI$ has the same length, $|I|=|gI|$. 
	
	Then, following \cite{pankov}, $\G$ equipped with the action of $\zz^n$ is said to be periodic if the action is 
	
	\begin{itemize}
		\item[1.] \textit{free}, that is, $gx=x\,\Longrightarrow\,g=0$;
		\item[2.] \textit{discrete}, that is, for every $x\in\G$, there is a neighbourhood $U$ of $x$ such that $gx\notin U$, for every $g\in\zz^n/\{0\}$;
		\item[3.] \textit{co-compact}, that is, there exists a compact set $Y\subset\G$ such that $\G=\bigcup_{g\in\zz^n}gY$.
	\end{itemize}  
	
	Particularly, the co-compactness implies that the whole graph $\G$ can be seen as the orbit through the action of $\zz^n$ of a fixed subset of $\G$, which is called a \textit{fundamental domain}.
	
	Therefore, the previous definition follows the strategy of identifying a periodic cell that repeats itself within a given graph $\G$ under a $\zz^n-$symmetry. However, for the purposes of the present paper, we decide to exploit the inverse direction, introducing a dual definition of periodicity which moves from a given compact graph and prescribes a way to glue together infinitely many copies of it to form a periodic structure. Such a procedure is actually equivalent to a special case in the general definition above, namely the case of periodic graphs sharing a $\zz$--symmetry (see Remark \ref{rem-def} below). 
	
	Let then $\K$ be a compact graph, i.e., a graph with a finite number of vertices and edges, all
	of finite length. Let $D$ (\textit{donors}) and $R$ (\textit{receivers}) be two non-empty subsets of the set $V(\K)$ of vertices of $\K$, and $\sigma:D\to R$ be a function (from donors to receivers) such that
	
	\begin{itemize}
		\item[\textit{(i)}] $D\cap R=\emptyset$;
		\item[\textit{(ii)}] $\sigma$ is bijective.
	\end{itemize}
	
	Consider now an infinite number of copies of $\K$, indexed by the integers, $\{\K_i\}_{i\in\zz}$, and let $D_i,R_i$ be the subsets of $V(\K_i)$ corresponding to $D,R$, respectively, for every $i\in\zz$. Setting $\mathbb{G}:=\bigcup_{i\in\zz}\K_i$, and thinking of $\sigma$ as a map from $D_i$ to $R_{i+1}$, for every $i$, we introduce the relation
	
	\[
	v\sim w\quad\Longleftrightarrow\quad\begin{cases}
	v=w & \text{if }v,w\in \K_i\,,\text{ for some }i\in\zz\\
	\sigma(v)=w & \text{if }v\in D_i\,,\, w\in R_{i+1}\,,\text{ for some }i\in\zz\\
	\sigma(w)=v & \text{if }v\in R_{i+1}\,,\, w\in D_i\,,\text{ for some }i\in\zz
	\end{cases}
	\]
	
	\noindent for every $v,w\in\mathbb{G}$.
	
	It is immediate to verify that the above relation is well-defined and it is in fact an equivalence on $\mathbb{G}$. We thus give the following definition.
	
	\begin{defn}
		\label{DEF-periodic}
		Let
		
		\[
		\G:=\mathbb{G}_{/\sim}
		\]
		
		\noindent denote the quotient space of $\mathbb{G}$ with respect to $\sim$. We say that $\G$ is a \textit{periodic graph} with \textit{periodicity cell} $\K$ and \textit{pasting rule} $\sigma$.
	\end{defn}
	
    \begin{rem}
        As a first example, note that also the real line $\rr$ can be seen as a periodic graph in the spirit of Definition \ref{DEF-periodic}, letting, for instance, $\K=[0,1]$, $D=\{1\}$, $R=\{0\}$ and $\sigma(1)=0$. However, as anticipated in the Introduction, to avoid such situation, we always assume that $\G$ has at least one vertex of degree at least 3.
    \end{rem}	

	As $\sigma$ does not involve the edges of $\K_i$, for any $i\in\zz$, the sets of vertices and
	edges of $\G$ are
	
	\[
	V(\G)=\Big(\bigcup_{i\in\zz}V(\K_i)\Big)_{/\sim}\qquad E(\G)=\bigcup_{i\in\zz}E(\K_i)\,.
	\]
	
	\noindent Moreover, we highlight that, for every $i$, the pasting rule $\sigma$ always maps $D_i$ into $R_{i+1}$. Henceforth, by construction, the only periodic graphs we are considering are the ones in which each periodicity cell shares connections with exactly two of the others, i.e., the graph enjoys a $\zz$-symmetry. 
	
	It is immediate to verify that Definition \ref{DEF-periodic} can be seen as a particular case of the one given in \cite{pankov} with $n=1$. Indeed, if $\G$ is periodic as in Definition \ref{DEF-periodic}, then every point of the graph belongs to a certain copy of the periodicity cell $\K$. Hence, given any $x\in\K$, let us denote by $x_i\in\G$ the corresponding point of $\G$ belonging to the $i$--th copy of $\K$, $\K_i$. According to this notation, it is readily seen that the group action given by
	
	\[
	\begin{split}
	\zz\times\G&\to\G\\
	(k,x_i)&\mapsto x_{k+i}\qquad\forall\,i\in\zz
	\end{split}
	\]
	
	\noindent  is free, discrete and co--compact and that $\K$ is a fundamental domain.
	
	\begin{rem}
	\label{rem-def}
	We can actually show that Definition \ref{DEF-periodic} is equivalent to the one in \cite{pankov} with $n=1$. 
	
	To this aim, assume that $\G$ is periodic according to \cite{pankov} with $n=1$ and let $Y$ be a fundamental domain satisfying the additional property that, for every $x\in Y$, we have $-1x\notin Y$ (being this not restrictive). Since $\G$ is connected, there must exist at least one point $x\in Y$ such that, for every neighbourhood $U\subset\G$ of $-1x$, then $U\cap Y\neq\emptyset$. Thus letting 
	
	\[
	\begin{split}
		D:&=\{\,x\in Y\,:\,U\cap Y\neq\emptyset\,,\,\forall\, U\text{ neighbourhood of }-1x\,\}\\
		R:&=\{\,y\notin Y\,:\,y=-1x\text{ for some }x\in D\,\}\\
		\K:&=Y\cup R\\
		\sigma:&\,D\to R\qquad\sigma (x)=-1x\,,
	\end{split}
	\]
	
	\noindent we get that $\G$ is the periodic graph arising from Definition \ref{DEF-periodic} with periodicity cell $\K$ and pasting rule $\sigma$.
	\end{rem}
	We end up this section by showing that properties \textit{(i)}-\textit{(ii)} we require introducing $D,R$ and $\sigma$ imply that diam$(\G)=+\infty$. 
	
	Indeed, let $x,y\in\K$ be such that $x\in D, y\in R$ and $\sigma(x)=y$. $\K$ being connected, let $\gamma\subset\K$ be the smallest path joining $x$ and $y$. Denote by $x_i,y_i\in\K_i$ the points corresponding to $x,y$ belonging to the $i-$th copy of $\K$, and by $\gamma_i\subset\K_i$ the corresponding copy of $\gamma$. As $\sigma(x_i)=y_{i+1}$, so that, building up $\G$ according to Definition \ref{DEF-periodic}, $x_i$ and $y_{i+1}$ becomes the same point, it follows that the union of all $\gamma_i$ is a connected path in $\G$, leading to
	
	\[
	\text{diam}(\G)\geq\big|\bigcup_{i\in\zz}\gamma_i\big|=\sum_{i\in\zz}|\gamma_i|=+\infty
	\]
	
	\noindent since $|\gamma_i|=|\gamma|>0$, for every $i\in\zz$.
	
	Let us stress the fact that assumptions \textit{(i)}-\textit{(ii)} are made to ease the above general definition, but it has to be shown that this does not raise any restriction on the class of graphs we are dealing with. We address this point throughout Appendix \ref{appendix}, where a wider discussion of the generality of Definition \ref{DEF-periodic} from the standpoint of graph theory is performed.

	\section{The subcritical regime $p\in(2,6)$}
	\label{sec:subcritical}
	
	\subsection{Preliminaries and compactness}
	\label{subsec:compactness}
	
	Before going on, let us briefly recall some known facts about the stationary nonlinear Schr\"odinger equation with a subcritical exponent that will play an important role in what follows.
	
	When $\G=\rr$, it is well-known (see for instance \cite{AST calc var}) that, for every mass $\mu>0$, there always exists a ground state, the \textit{soliton} $\phi_\mu$, given by
	
	\begin{equation}
		\label{EQ-phi mu def}
		\phi_\mu(x)=\mu^\alpha \phi_1(\mu^\beta x)\,,\qquad\alpha=\frac{2}{6-p},\quad\beta=\frac{p-2}{6-p}
	\end{equation}
	
	\noindent where $\phi_1$ is defined as
	
	\begin{equation}
		\label{EQ-phi 1 def}
		\phi_1(x)=A_p\text{sech}^{\alpha/\beta}(a_p x)
	\end{equation}
	
	\noindent with $A_p,a_p>0$. Moreover, $E(\phi_\mu,\rr)<0$, for every $\mu>0$.
	
	If $\G=\rr^+$ the behaviour is the same, with solitons replaced by the \textit{half-solitons}, namely their restrictions to the half-line.

	The forthcoming analysis makes use also of the Gagliardo-Nirenberg inequality
	
	\begin{equation}
		\label{EQ-GN p}
		\uLp^p\leq C_{\G,p}\uLtwo^{\frac{p}{2}+1}\udot^{\frac{p}{2}-1}
	\end{equation}
	
	\noindent holding for every non-compact graph $\G$, $u\in H^1(\G)$ and $p\geq 2$ (we refer to \cite{AST funct an} for further details). Here $C_{\G,p}>0$ depends only on $\G$ and $p$.
	We state now the following result, proving that globally minimizing sequences of the NLS
	energy functional are strongly compact in $\HmuG$ up to translations, whenever the infimum of
	\eqref{EQ-def energy intro} is strictly negative.
	
	\begin{prop}
		\label{PROP-compactness}
		Let $\G$ be a periodic graph, $p\in(2,6)$ and $\mu>0$. Let $\{u_n\}_{n\in\nn}\subset\HmuG$ be a minimizing sequence for $E$ such that, for every $n\in\nn$, there exists $x_n\in\K_0$ so that $\|u_n\|_{L^\infty(\G)}=u_n(x_n)$. If \eqref{EQ-inf negative THM 1} holds, then there exists $u\in\HmuG$ such that $u_n\to u$ strongly in $H^1(\G)$ and $u$ is a ground state.
	\end{prop}

	\begin{proof}
		Let $\{u_n\}_{n\in\nn}$ be a minimizing sequence for $E$ in $\HmuG$ as above. Plugging Gagliardo-Nirenberg inequality \eqref{EQ-GN p} in \eqref{EQ-def energy intro}, we have
		
		\begin{equation}
			\label{EQ-bound on energy GN}
			E(u_n,\G)\geq\frac{1}{2}\|u_n'\|_{L^2(\G)}^2-\frac{C_{\G,p}}{p}\mu^{\frac{p}{4}+\frac{1}{2}}\|u_n'\|_{L^2(\G)}^{\frac{p}{2}-1}=\frac{1}{2}\|u_n'\|_{L^2(\G)}^2\Big(1-\frac{2C_{\G,p}}{p}\mu^{\frac{p}{4}+\frac{1}{2}}\|u_n'\|_{L^2(\G)}^{\frac{p}{2}-3}\Big)
		\end{equation}
		
		\noindent and, since $p\in(2,6)$, this implies that $\{u_n\}_{n\in\nn}$ is bounded in $H^1(\G)$. Hence, there exists $u\in H^1(\G)$ so that (up to subsequences) $u_n\rightharpoonup u$ in $H^1(\G)$ and $u_n\to u$ in $L_{loc}^\infty(\G)$. Moreover, by weak lower semicontinuity,
		
		\begin{equation}
			\label{EQ-weak lower sem}
			\udot\leq\liminf_{n\to+\infty}\|u_n'\|_{L^2(\G)}\quad\text{and}\quad\uLtwo\leq\liminf_{n\to+\infty}\|u_n\|_{L^2(\G)}=\sqrt{\mu}\,.
		\end{equation}
		
		\noindent We first prove that $u\not\equiv0$. Suppose, by contradiction, $u\equiv 0$. Then, since $\|u_n\|_{L^\infty(\G)}=u_n(x_n)\to0$ when $n\to+\infty$, as $x_n\in\K_0$ for every $n\in\nn$,
		
		\begin{equation}
			\label{EQ-u^p to 0 with u  infty}
			\|u_n\|_{L^p(\G)}^p\leq\|u_n\|_{L^\infty(\G)}^{p-2}\mu\to0\qquad\text{for }n\to+\infty
		\end{equation}
		
		\noindent and 
		
		\[
		0>\mathcal{E}_\G(\mu)=\lim_{n\to+\infty}E(u_n,\G)\geq-\lim_{n\to+\infty}\frac{1}{p}\|u_n\|_{L^p(\G)}^p=0
		\]
		
		\noindent provides the contradiction we seek.
		
		Thus, either $0<\uLtwo^2<\mu$ or $\uLtwo^2=\mu$. If the latter case occurs, then $u\in\HmuG$, $u$ is a ground state of $E$ and $u_n\to u$ strongly in $H^1(\G)$. Let us thus prove that the former never happens, adapting the argument in the proof of Lemma 3.2 in \cite{ADST}.
		
		Suppose by contradiction $\uLtwo^2=:m<\mu$. By the Brezis-Lieb Lemma \cite{brezislieb}, we have, for $n$ sufficiently large
		
		\begin{equation}
			\label{EQ-brezislieb split}
			E(u_n,\G)=E(u_n-u,\G)+E(u,\G)+o(1)
		\end{equation}
		
		\noindent and by weak convergence in $L^2(\G)$ of $u_n$ to $u$,
		
		\begin{equation}
			\label{EQ-mass u_n-u}
			\begin{split}
			\|u_n-u\|_{L^2(\G)}^2=&\|u_n\|_{L^2(\G)}^2+\uLtwo^2-2<u_n,u>_{L^2(\G)}\\
			=&\mu-\uLtwo^2+o(1)=\mu-m+o(1)\,.
			\end{split}
		\end{equation}
		
		\noindent Therefore, we have
		
		\[
		\begin{split}
		\mathcal{E}_\G(\mu)\leq&E\Big(\frac{\sqrt{\mu}}{\|u_n-u\|_{L^2(\G)}}(u_n-u),\G\Big)\\
		=&\frac{1}{2}\frac{\mu}{\|u_n-u\|_{L^2(\G)}^2}\|u_n'-u'\|_{L^2(\G)}^2-\frac{1}{p}\frac{\mu^{\frac{p}{2}}}{\|u_n-u\|_{L^2(\G)}^p}\|u_n-u\|_{L^p(\G)}^p\\
		=&\frac{\mu}{\|u_n-u\|_{L^2(\G)}^2}\Big(\frac{1}{2}\|u_n'-u'\|_{L^2(\G)}^2-\frac{1}{p}\frac{\mu^{\frac{p}{2}-1}}{\|u_n-u\|_{L^2(\G)}^{p-2}}\|u_n-u\|_{L^p(\G)}^p\Big)\\
		<&\frac{\mu}{\|u_n-u\|_{L^2(\G)}}E(u_n-u,\G)
		\end{split}
		\]
		
		\noindent with the last inequality coming from the fact that $\|u_n-u\|_{L^2(\G)}^2<\mu$. Taking the liminf and combining with \eqref{EQ-mass u_n-u}, we get
		
		\begin{equation}
			\label{EQ-liminf geq inf}
			\liminf_{n\to+\infty}E(u_n-u,\G)\geq\frac{\mu-m}{\mu}\mathcal{E}_\G(\mu)\,.
		\end{equation}
		
		\noindent Moreover, similar calculations lead to
		
		\[
		\mathcal{E}_\G(\mu)\leq E\Big(\sqrt{\frac{\mu}{m}}u,\G\Big)=\frac{\mu}{m}\Big(\frac{1}{2}\udot^2-\frac{1}{p}\frac{\mu^{\frac{p}{2}-1}}{m^{\frac{p}{2}-1}}\uLp^p\Big)<\frac{\mu}{m}E(u,\G)
		\]
		
		\noindent that is
		
		\begin{equation}
			\label{EQ-E u >inf}
			E(u,\G)>\frac{m}{\mu}\mathcal{E}_\G(\mu)\,.
		\end{equation}
		
		\noindent Now, considering the liminf in \eqref{EQ-brezislieb split} and combining with \eqref{EQ-liminf geq inf}-\eqref{EQ-E u >inf}, we end up with
		
		\[
		\mathcal{E}_\G(\mu)=\liminf_{n\to+\infty}E(u_n,\G)=\liminf_{n\to+\infty}E(u_n-u,\G)+E(u,\G)>\frac{\mu-m}{\mu}\mathcal{E}_\G(\mu)+\frac{m}{\mu}\mathcal{E}_\G(\mu)=\mathcal{E}_\mu(\G)
		\]
		
		\noindent and we conclude.
	\end{proof}

	\subsection{Proof of Theorem \ref{THM 1 INTRO}}
	
	Note first that, given $\mu>0$, if we manage to prove that the infimum of the energy is strictly negative, then the existence of ground states is straightforward. Indeed, if \eqref{EQ-inf negative THM 1} holds, namely $\mathcal{E}_\G(\mu)<0$, then the statement of the theorem immediately follows from Proposition \ref{PROP-compactness}. Furthermore, it is readily seen that, for every $\mu>0$
	
	\[
	\mathcal{E}_\G(\mu)>-\infty\,.
	\]
	
	\noindent Indeed, \eqref{EQ-bound on energy GN} provides a lower bound for $E(u,\G)$ showing that $E(u,\G)\to+\infty$ if $\udot\to+\infty$, that is, $E$ is bounded from below for every $\mu>0$.
	
	We are thus left to prove that $\mathcal{E}_\G(\mu)$ is always negative, for every fixed $\mu$ and $\G$. As usual, let $\K$ be the periodicity cell of $\G$. We introduce
	
	\[
	L(\K):=\{e\in E(\K)\,:\,\exists\,!\,v\in D(\K)\,\text{such that }e\succ v\}
	\]
	
	\noindent as the set of edges of $\K$ with exactly one endpoint in $D(\K)$. Moreover, for every $e\in L(\K)$, we define the coordinate $x_e$ on $e$ so that, if $e\succ v$ and $v\in D(\K)$, then $x_e(0)=v$. Denote by $\ell:= \min_{e\in L(\K)}|e|$ the length of the smallest edge of $L(\K)$, and define
	
	\[
	\widetilde{\K}=\K-\bigcup_{e\in L(\K)}(e\cap[0,\ell])
	\]
	
	\noindent the portion of $\K$ that is left when we get rid of a path along each $e\in L(\K)$ of length $\ell$ starting at $x_e(0)$. Note that there may exist edges belonging to $\widetilde{\K}$ joining vertices in $D(\K)$. For the sake of simplicity, let us assume in the remainder of the proof that there is no such edge in $\widetilde{\K}$. Since all the forthcoming constructions straightforwardly generalize in the presence of this kind of edges, this does not reflect into any loss of generality, but it helps in simplifying some notation. 
	
	
	\begin{figure}[t]
		\centering
		\subfloat[][]{
			\includegraphics[width=0.4\columnwidth]{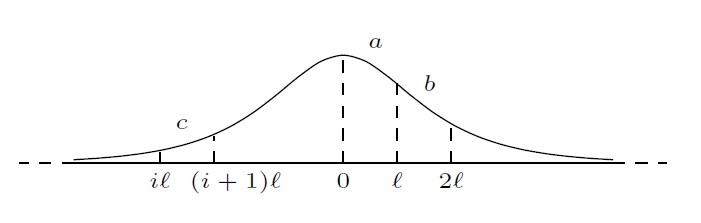}
		}\qquad
		\subfloat[][]{
			\includegraphics[width=0.5\columnwidth]{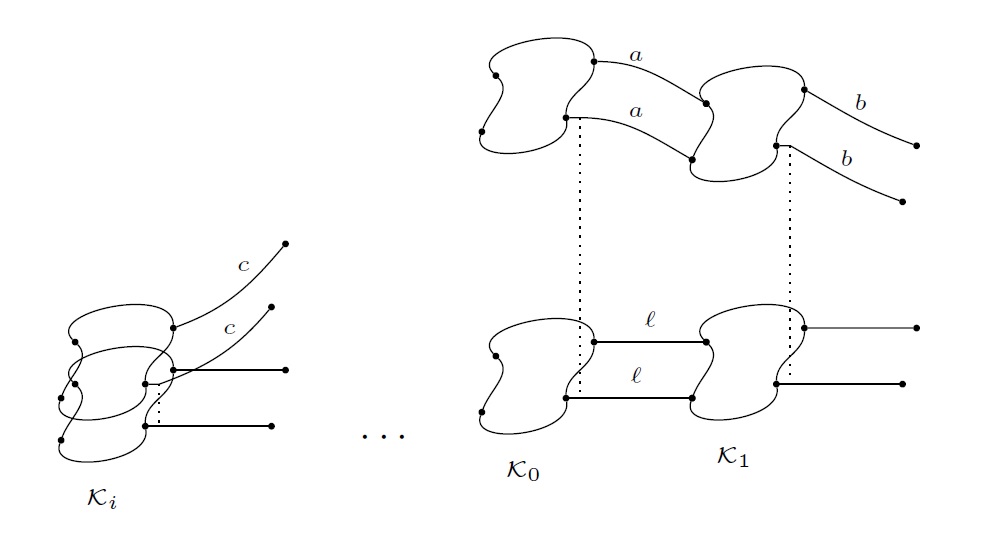}}
		\caption{(a) the soliton $\phi_{\mu_1}$ on the real line and (b) the corresponding $u$ on a periodic graph as in the proof of Theorem \ref{THM 1 INTRO}.}
		\label{FIG-soliton}
	\end{figure}
	
	
	We then define the function
	
	\begin{equation}
		\label{EQ-u negative}
		u(x):=\begin{cases}
		\phi_{\mu_1}(x-i\ell) & \text{if }x\in e\cap[0,\ell],\,\text{for some }e\in L(\K_i),\, \text{for some }i\in\{-1,-2,...\}\\
		\phi_{\mu_1}(-(i+1)\ell) & \text{if }x\in\widetilde{\K}_i,\,\text{for some }i\in\{-1,-2,...\}\\
		\phi_{\mu_1}((i-1)\ell) & \text{if }x\in\widetilde{\K}_i,\,\text{for some }i\in\{0,1,...\}\\
		\phi_{\mu_1}((i-1)\ell-x) & \text{if }x\in e\cap[0,\ell],\,\text{for some }e\in L(\K_i),\,\text{for some }i\in\{0,1,...\}
		\end{cases}
	\end{equation}
	
	\noindent where $\phi_{\mu_1}\in H_{\mu_1}^1(\rr)$ is the soliton on $\rr$ of mass $\mu_1$ as in \eqref{EQ-phi mu def}, for some $\mu_1\in(0,\mu)$ that has to be chosen to ensure $\uLtwo^2=\mu$.
	
	Note that, setting $m:=|L(\K)|$, we have, by construction
	
	\begin{equation*}
		\label{EQ-energy soliton on G}
		E(u,\G)=mE(\phi_{\mu_1},\rr)+\sum_{i\in\zz}E(u,\widetilde{\K}_i)=mE(\phi_{\mu_1},\rr)-\frac{1}{p}\sum_{i\in\zz}\|u\|_{L^p(\widetilde{\K}_i)}^p<0\,.
	\end{equation*}
	
	\noindent Therefore, to conclude, let us show that for every $\mu>0$ there exists
	$\mu_1\in(0,\mu)$ such that $u$ as in \eqref{EQ-u negative} belongs to $\HmuG$.
	
	Denoting by $\Gamma:=|\widetilde{\K}|$, we have
	
	\begin{equation}
		\label{EQ-massu negative}
		\uLtwo^2=m\mu_1^2+\Gamma\sum_{i\in\zz}\phi_{\mu_1}^2(i\ell)=m\mu_1+\Gamma\Big(2\sum_{i=0}^{+\infty}\phi_{\mu_1}^2(i\ell)-\mu_1^{2\alpha}\Big)\,.
	\end{equation}
	
	\noindent Using the explicit formulas \eqref{EQ-phi mu def}-\eqref{EQ-phi 1 def}, observe that (up to some constant) $\phi_1^2(x)\sim e^{-2\frac{\alpha}{\beta}x}$, for $x$ large enough.
	
	Thus, 
	
	\[
	\sum_{i=0}^{+\infty}\phi_{\mu_1}^2(i\ell)\sim\mu_1^{2\alpha}\sum_{i=0}^{+\infty}e^{-2\frac{\alpha}{\beta}\mu_1^\beta i\ell}=\mu_1^{2\alpha}\frac{e^{2\frac{\alpha}{\beta}\mu_1^\beta \ell}}{e^{2\frac{\alpha}{\beta}\mu_1^\beta\ell}-1}\,.
	\]
	
	\noindent Plugging into \eqref{EQ-massu negative}, we get that $\uLtwo$ is a continuous function of $\mu_1$, $\uLtwo=0$ if $\mu_1=0$ and $\uLtwo\to+\infty$ as $\mu_1\to+\infty$, and we conclude.
	
	Dropping the assumption that no edge in $\widetilde{\K}$ joins vertices in $D(\K)$ only reflects in minor modifications of the above argument. Indeed, if $e\in\widetilde{\K}$ is an edge between $v,w\in D(\K)$, we modify \eqref{EQ-u negative} defining $u(x) = u(v_i) = u(w_i)$, for every $x\in e_i$ and every $i\in\zz$, where $e_i, v_i, w_i$
	denotes the copies of $e, v, w$ in $\K_i$ respectively. Keeping track of this new definition, all the previous calculations can be developed in the same way.
	\hspace{\stretch{1}} $\Box$

	\section{The critical regime $p=6$}
	\label{sec:critical}
	
	Let us focus now on the critical setting. Recall that, when $p=6$ (see \cite{AST critico}), a threshold value of the mass can be defined
	
	\begin{equation}
		\label{EQ-critical mass def}
		\mu_\G:=\sqrt{\frac{3}{C_\G}}\,,
	\end{equation}
	
	\noindent with $C_\G$ denoting the optimal constant in the Gagliardo-Nirenberg inequality
	
	\begin{equation}
		\label{EQ-critical GN}
		\uLsix^6\leq C_\G\uLtwo^4\udot^2\,.
	\end{equation}
	
	\noindent Plugging \eqref{EQ-critical GN} into the energy \eqref{EQ-def energy intro}, we get
	
	\begin{equation}
		\label{EQ-bound p=6 energy GN}
		E(u,\G)\geq\frac{1}{2}\udot^2\Big(1-\frac{C_\G}{3}\mu^2\Big)
	\end{equation}
	
	\noindent implying that
	
	\[
	\begin{split}
	\mu&\leq\mu_\G\quad\Longrightarrow\quad E(u,\G)\geq0,\,\forall\, u\in\HmuG\\
	\mu&>\mu_\G\quad\Longrightarrow\quad \exists\,u\in\HmuG\text{ such that }E(u,\G)<0\,.
	\end{split}
	\]
	
	\noindent If $\G=\rr$, then $\mu_\rr=\frac{\sqrt{3}}{2}\pi$,
	
	\begin{equation}
		\label{EQ-inf critical R}
		\mathcal{E}_\rr(\mu)=\begin{cases}
		0 & \text{if }\mu\leq\mu_\rr\\
		-\infty & \text{if }\mu>\mu_\rr
		\end{cases}
	\end{equation}
	
	\noindent and a whole family of critical solitons $\{\phi_\lambda\}_{\lambda>0}$ exists if and only if $\mu=\mu_\rr$, given by
	
	\begin{equation}
		\label{eq-crit sol}
		\phi_\lambda(x):=\sqrt{\lambda}\sqrt{\text{sech}\Big(\frac{2}{\sqrt{3}}\lambda x\Big)}\,.
	\end{equation}
	
	\noindent When $\G=\rr^+$, nothing changes, except of $\mu_{\rr^+}=\frac{\sqrt{3}}{4}\pi$ and ground states being the restriction of $\{\phi_\lambda\}_{\lambda>0}$ to the half-line.
	
	For a general non-compact metric graph $\G$, it is known (see Proposition 2.4 in \cite{AST critico}) that
	
	\begin{equation}
		\label{EQ-mu between rr and rr+}
		\mu_{\rr^+}\leq\mu_\G\leq\mu_\rr\,.
	\end{equation}
	
	\noindent Even though the original proof is developed for graphs with half-lines, it extends without any modifications to the periodic graphs we are dealing with.
	
	The following proposition provides a first topological characterization of $\mu_\G$ for periodic graphs.
	
	\begin{prop}
		\label{PROP-mu G periodic}
		Let $\G$ be a periodic graph. It holds that
		\begin{itemize}
			\item[(i)] if $\G$ satisfies (H$_{per}$), then $\mu_\G=\mu_\rr$;
			\item[(ii)] if $\G$ has a terminal edge, then $\mu_\G=\mu_{\rr^+}$.
		\end{itemize}
	\end{prop}

	\begin{proof}
		Let us first deal with part \textit{(i)}. Note that, since $\G$ satisfies assumption (H$_{per}$), then, for every $u\in\HmuG$ and almost every $t$ in the image of $u$, we have
		
		\[
		\symbol{35}\{x\in\G\,:\,u(x)=t\}\geq2\,,
		\]
		
		\noindent that is, any value in the image of $u$ has at least two pre-images on $\G$. Indeed, let $M := \|u\|_{L^\infty(\G)}$ and $\overline{x}\in\G$ be such that $u(\overline{x})=M$. Then, by (H$_{per}$), there exist two disjoint paths of infinite
		length originating at $\overline{x}$, say $\Gamma_1,\Gamma_2$, and since $u\in H^1(\G)$, $u(\Gamma_1)=u(\Gamma_2)=(0,M)$.
		
		Hence, by standard properties of symmetric rearrangements (see \cite{AST calc var}), denoting by $\widehat{u}\in H^1(\rr)$ the symmetric rearrangement of $u\in H^1(\G)$ on the line, it follows
		
		\[
		\frac{\uLsix^6}{\uLtwo^4\udot^2}\leq\frac{\|\widehat{u}\|_{L^6(\rr)}^6}{\|\widehat{u}\|_{L^2(\rr)}^4\|(\widehat{u})'\|_{L^2(\rr)}^2}\leq C_\rr
		\]
		
		\noindent for every $u\in H^1(\G)$, and taking the supremum
		
		\[
		C_\G\leq C_\rr\,.
		\]
		
		\noindent By \eqref{EQ-critical mass def}, this means
		
		\[
		\mu_\G\geq\mu_\rr
		\]
		
		\noindent and, combining with \eqref{EQ-mu between rr and rr+}, we conclude.
		
		Let us focus now on statement \textit{(ii)}. For every $\varepsilon > 0$, there exists $u\in H^1(\rr^+)$, supported on $[0, 1]$, so that
		
		\[
		\frac{\|u\|_{L^6(\rr^+)}^6}{\|u\|_{L^2(\rr^+)}^4\|u'\|_{L^2(\rr^+)}^2}> C_{\rr^+}-\varepsilon\,.
		\]
		
		\noindent Setting $u_\lambda(x):=\sqrt{\lambda}u(\lambda x)$, for every $x\in\rr^+$ and $\lambda>0$, we have that $u_\lambda\in H^1(\rr^+)$ and
		
		\[
		\begin{split}
		\text{supp}\,u_\lambda=&\big[0,\frac{1}{\lambda}\big]\\
		\frac{\|u_\lambda\|_{L^6(\rr^+)}^6}{\|u_\lambda\|_{L^2(\rr^+)}^4\|u_\lambda'\|_{L^2(\rr^+)}^2}=&\frac{\|u\|_{L^6(\rr^+)}^6}{\|u\|_{L^2(\rr^+)}^4\|u'\|_{L^2(\rr^+)}^2}\,.
		\end{split}
		\]
		
		\noindent Hence, letting $e\in E(\G)$ be a terminal edge of $\G$ and $\ell:=|e|$ its length, when $\lambda$ is large enough, $\text{supp}\,u_\lambda\subset e$, and defining $v\in H^1(\G)$ so that $v\equiv u_\lambda$ on $e$ and $v\equiv 0$ elsewhere, we
		deduce
		
		\[
		C_\G>C_{\rr^+}-\varepsilon
		\]
		
		\noindent and, by the arbitrariness of $\varepsilon$,
		
		\[
		C_\G\geq C_{\rr^+}\,.
		\]
		
		\noindent Combining with \eqref{EQ-critical mass def} and \eqref{EQ-mu between rr and rr+} gives the claim.
		
	\end{proof}
	
	As a direct consequence, we are now able to prove the first of our main results in the
	critical setting.
	
	\begin{proof}[Proof of Theorem \ref{THM 2 INTRO}] We begin by proving statement \textit{(i)}. Let $\G$ be a periodic graph of periodicity cell $\K$ satisfying (H$_{per}$).
	
	By Proposition \ref{PROP-mu G periodic}\textit{(i)}, $\mu_\G=\mu_\rr$, and \eqref{EQ-bound p=6 energy GN} ensures that, for every $\mu\leq\mu_\rr$ and $u\in\HmuG$
	
	\begin{equation}
		\label{EQ-E geq 0}
		E(u,\G)\geq0\,.
	\end{equation}
	
	\noindent Now, for every $n\in\nn$, we introduce 
	
	\begin{equation}
		\label{EQ-def Sigma}
		\Sigma_n:=\{e\in E(\K_{-n-1})\cup E(\K_{n+1})\,:\,\exists\,v\in R(\K_{-n})\cup D(\K_n)\text{ such that }e\succ v\}
	\end{equation}
	
	\noindent as the set of all edges in $\G$ entering a vertex which is joining either $\K_{-n-1}$ with $\K_{-n}$ or $\K_{n+1}$ with $\K_n$. Moreover, for every $e\in\Sigma_n$ and $e\succ v\in R(\K_{-n})\cup D(\K_n)$, we set $x_e(0)=v$, being $x_e$ the coordinate defined on $e$. Let also $\ell_e:=|e|$ denote the length of $e$.
	
	Then, for every $n\in\nn$, we define $u_n\in\HmuG$ as
	
	\begin{equation}
		\label{EQ-def u n to zero energy}
		u_n(x):=\begin{cases}
		\alpha_n & \text{if }x\in\K_i,\text{ for some }i\in\{-n,...,n\}\\
		\frac{\alpha_n}{\ell_e}(\ell_e-x) & \text{if }x\in e,\text{ for some }e\in\Sigma_n\\
		0 & \text{otherwise on }\G
		\end{cases}
	\end{equation}
	
	\noindent where $\alpha_n$ is chosen to satisfy $\|u_n\|_{L^2(\G)}^2=\mu$. It is immediate to see that $\alpha_n\to0$ as $n\to+\infty$, thus implying $E(u_n,\G)\to0$ and
	
	\[
	\mathcal{E}_\G(\mu)=0
	\]
	
	\noindent for every $\mu\leq\mu_\rr$.
	
	When $\mu>\mu_\rr$, by \eqref{EQ-inf critical R} there exists $v\in H_\mu^1(\rr)$, supported on $[0,1]$, so that $E(v,\rr)<0$. Considering the mass-preserving transformation
	
	\[
	v_\lambda(x):=\sqrt{\lambda}v(\lambda x)
	\]
	
	\noindent for every $\lambda>0$, we get $v_\lambda\in H_\mu^1(\rr)$, $v_\lambda$ is supported on $[0,1/\lambda]$ and $E(v_\lambda,\rr)=\lambda^2 E(v,\rr)$.
	
	Fix now any edge $e\in E(\G)$ and let $\ell:=|e|$ be its length. Then, there exists $\overline{\lambda}>0$ such
	that, for $\lambda\geq\overline{\lambda}$, $v_\lambda\in H_\mu^1(0,\ell)$, that is, we construct functions $\{v_\lambda\}_{\lambda\geq\overline{\lambda}}\subset\HmuG$, supported
	on $e$ and satisfying
	
	\[
	E(v_\lambda,\G)\to-\infty\qquad\text{for }\lambda\to+\infty\,,
	\]
	
	\noindent proving that, for every $\mu>\mu_\rr$
	
	\[
	\mathcal{E}_\G(\mu)=-\infty\,.
	\]
	
	\noindent We conclude showing that the infimum is not attained, for any value of the mass $\mu\leq\mu_\rr$ (the statement is trivially true in the regime $\mu>\mu_\rr$).
	
	When $\mu<\mu_\rr$, the result is immediate, the inequality in \eqref{EQ-E geq 0} being strict for every $u\in\HmuG$.
	
	If $\mu=\mu_\rr$, suppose by contradiction that $u\in H_{\mu_\rr}^1(\G)$ is a ground state, i.e. $E(u,\G)=\mathcal{E}_{\G}(\mu_\rr)=0$. This implies
	
	\[
	0=E(u,\G)\geq E(\widehat{u},\rr)\geq \mathcal{E}_\rr(\mu_\rr)=0\,,
	\]
	
	\noindent that is, $E(u,\G)=E(\widehat{u},\rr)$ and particularly $\udot=\|\widehat{u}'\|_{L^2(\rr)}$. But this is impossible, since $\G$ contains at least one vertex of degree 3, preventing $\symbol{35}\{x\in\G\,:\,u(x)=t\}=2$ to be true for a.e. $t$ in the image of $u$.
	
	Part \textit{(ii)} of Theorem \ref{THM 2 INTRO} can be proved by the same argument, simply replacing $\mu_\rr$ with $\mu_{\rr^+}$ and symmetric rearrangements with decreasing ones whenever needed.
	\end{proof}

	We then turn our attention to graphs violating (H$_{per}$) and with no terminal edge, providing the proof of Theorem \ref{THM 3 INTRO}. To this purpose, a modified version of the Gagliardo-Nirenberg inequality has to be considered, ensuring that, for every $\mu\in[0,\mu_\rr]$ and every $u\in\HmuG$, there exists $\theta_u:=\theta(u)$, with $\theta_u\in[0,\mu]$, such that
	
	\begin{equation}
		\label{EQ-modified GN}
		\uLsix^6\leq3\Big(\frac{\mu-\theta_u}{\mu_\rr}\Big)^2\udot^2+C\sqrt{\theta_u}
	\end{equation}
	
	\noindent with $C>0$ depending only on $\G$ (see Lemma 4.4 in \cite{AST critico} for a proof that extends to periodic graphs without modifications).
	
	\medskip
	\begin{proof}[Proof of Theorem \ref{THM 3 INTRO}] By the same argument in the proof of Theorem \ref{THM 2 INTRO}, we have
	
	\[
	\begin{split}
	\mathcal{E}_\G(\mu)=&-\infty\qquad\,\forall\mu>\mu_\rr\\
	\mathcal{E}_\G(\mu)=&0\qquad\qquad\forall\mu<\mu_\G	\end{split}
	\]
	
	\noindent and ground states do not exist in both these situations.
	
	Consider now $\mu\in(\mu_\G,\mu_\rr]$. For every $\varepsilon>0$, there exists $u\in\HmuG$ so that
	
	\[
	\frac{\uLsix^6}{\uLtwo^4\udot^2}>C_\G-\varepsilon
	\]
	
	\noindent and plugging into the energy
	
	\[
	E(u,\G)\leq\frac{1}{2}\udot^2\Big(1-\frac{C_\G-\varepsilon}{3}\mu^2\Big)\,.
	\]
	
	\noindent Thus, picking $\varepsilon$ small enough and since $\mu>\mu_\G$
	
	\begin{equation}
		\label{EQ-inf negative}
		\mathcal{E}_\G(\mu)<0\qquad\forall\mu\in(\mu_\G,\mu_\rr]\,.
	\end{equation}
	
	\noindent Let $\{u_n\}_{n\in\nn}\subset\HmuG$ be a minimizing sequence for $E$, so that each $u_n$ realizes its $L^\infty$ norm at some point of $\K_0$. Then, by \eqref{EQ-inf negative} and \eqref{EQ-modified GN}, we have
	
	\[
	\|u_n'\|_{L^2(\G)}^2<\frac{1}{3}\|u_n\|_{L^6(\G)}^6\leq\Big(\frac{\mu-\theta_{u_n}}{\mu_\rr}\Big)^2\|u_n'\|_{L^2(\G)}^2+C\sqrt{\theta_{u_n}}\leq\Big(\frac{\mu-\theta_{u_n}}{\mu_\rr}\Big)^2\|u_n'\|_{L^2(\G)}^2+C\sqrt{\mu}
	\]
	
	\noindent that is 
	
	\begin{equation}
		\label{EQ-bound u'}
		\Big(1-\frac{(\mu-\theta_{u_n})^2}{\mu_\rr^2}\Big)\|u_n'\|_{L^2(\G)}^2\leq C\sqrt{\mu}\,.
	\end{equation}
	
	\noindent On the other hand, plugging \eqref{EQ-modified GN} into \eqref{EQ-def energy intro}
	
	\begin{equation}
	\label{EQ-lower bound E proof 1.3}
	E(u_n,\G)\geq\frac{1}{2}\|u_n'\|_{L^2(\G)}^2\Big(1-\frac{(\mu-\theta_{u_n})^2}{\mu_\rr^2}\Big)-C\sqrt{\theta_{u_n}}
	\end{equation}
	
	\noindent and, since the right-hand side tends to 0 when $\theta_{u_n}\to0$, combining with $\eqref{EQ-bound u'}$ and the minimality of $\{u_n\}_{n\in\nn}$ gives
	
	\begin{equation}
		\label{EQ-inf theta not 0}
		\inf_{n\in\nn}\theta_{u_n}>0\,.
	\end{equation}
	
	\noindent Hence, \eqref{EQ-bound u'}-\eqref{EQ-inf theta not 0} ensure that $\{u_n\}_{n\in\nn}$ is bounded in $H^1(\G)$ and $u_n\rightharpoonup u$, for some $u\in H^1(\G)$, whereas \eqref{EQ-lower bound E proof 1.3} guarantees that $\mathcal{E}_\G(\mu)>-\infty$. The argument of the proof of Proposition \ref{PROP-compactness} can now be repeated, showing that $u_n\to u$ strongly in $\HmuG$ and $u$ is a ground state of mass $\mu$.
	
	It remains to deal with the case $\mu=\mu_\G$. Note that this is not immediate, since now
	$\mathcal{E}_\G(\mu_\G)=0$, while the negativity of the ground state energy level is crucial in the above argument. Actually, it is no longer true that every minimizing sequence is strongly compact, vanishing sequences as in \eqref{EQ-def u n to zero energy} providing an example of the possible loss of compactness. However, we show that there exists a proper choice of the minimizing sequence recovering compactness.
	
	Let $\{u_n\}_{n\in\nn}\subset H_{\mu_\G}^1(\G)$ be a maximizing sequence for the Gagliardo-Nirenberg inequality \eqref{EQ-critical GN}, i.e.,
	
	\begin{equation}
		\label{EQ-quotient to 3}
		\lim_{n\to+\infty}\frac{\|u_n\|_{L^6(\G)}^6}{\|u_n'\|_{L^2(\G)}^2}\to C_\G \mu_\G^2=3
	\end{equation}
	
	\noindent and assume as usual that, for every $n$, $u_n$ attains its maximum somewhere in $\K_0$. It is straightforward to see that $\{u_n\}_{n\in\nn}$ is a minimizing sequence for $E$, being
	
	\[
	E(u_n,\G)=\frac{1}{2}\|u_n'\|_{L^2(\G)}^2\Big(1-\frac{\|u_n\|_{L^6(\G)}^6}{3\|u_n'\|_{L^2(\G)}^2}\Big)\to0\qquad\text{as }n\to+\infty\,.
	\]
	
	\noindent Now, by \eqref{EQ-modified GN}, we have
	
	\[
	\begin{split}
	3\|u_n'\|_{L^2(\G)}^2=\|u_n\|_{L^6(\G)}^6+o(1)\leq& 3\Big(\frac{\mu_\G-\theta_{u_n}}{\mu_\rr}\Big)^2\|u_n'\|_{L^2(\G)}^2+C\sqrt{\theta_{u_n}}+o(1)\\
	\leq&3\frac{\mu_\G^2}{\mu_\rr^2}\|u_n'\|_{L^2(\G)}^2+C\sqrt{\mu_\G}+o(1)
	\end{split}
	\]
	
	\noindent that is
	
	\[
	3\Big(1-\frac{\mu_\G^2}{\mu_\rr^2}\Big)\|u_n'\|_{L^2(\G)}^2\leq C\sqrt{\mu_\G}\,.
	\]
	
	\noindent As $\mu_\G<\mu_\rr$, this means that $\{u_n\}_{n\in\nn}$ is bounded in $H^1(\G)$, $u_n\rightharpoonup u$ and $u_n\to u$ in $L_{loc}^\infty(\G)$ as $n\to+\infty$, for some $u\in H^1(\G)$.
	
	Suppose $u\equiv 0$. Then $u_n\to0$ in $L^\infty(\G)$ and $u_n\to0$ in $L^p(\G)$, for every $p>2$. Therefore, by \eqref{EQ-quotient to 3}, $\|u_n'\|_{L^2(\G)}\to0$ for $n\to+\infty$.
	
	Let now $\mathcal{B}$ denote the set of all edges $e\in E(\K_0)$ such that, if removed, give birth to at least one compact connected component. Since $\G$ violates (H$_{per}$), $\mathcal{B}\not=0$. Note that, if there exist values in the image of $u_n$ with just one pre-image on $\G$, then they can only be attained somewhere in $\mathcal{B}$. Indeed, if $t<\min_{x\in\K_0}u_n(x)$, then $t$ has at least two pre-images, one on the
	left of $\K_0$ and one on the right. Moreover, if $t$ is attained at a point belonging to a cycle of $\K_0$, then it is attained twice on that cycle and it has at least two pre-images in $\K_0$.
	
	Let us introduce the following construction. We first double the edges of $\mathcal{B}$, that is, we replace each $e\in\mathcal{B}$ by two edges, say $e_1,e_2$, joining the same vertices of $e$ and so that $|e_1|=|e_2|=2|e|$. Then, both on $e_1$ and $e_2$, we stretch the restriction of $u_n$ to $e$ by a factor 2. Letting $\widetilde{\G}$ and $\widetilde{u}_n$ denote respectively the new graph and function identified by this procedure, we have, for every $n\in\nn$, $\widetilde{u}_n\in H^1(\widetilde{\G})$ and
	
	\[
	\begin{split}
	\widetilde{u}_n(x)=&u_n(x)\quad\qquad\text{if }x\in\G/\mathcal{B}\\
	\widetilde{u}_n(x)=&u_n(x/2)\qquad\text{if }x\in e_i\text{ for some }i\in\{1,2\}\text{ and }e\in\mathcal{B}\,.
	\end{split}
	\]
	
	\noindent Note that
	
	\begin{equation}
		\label{EQ-u' doubled}
		\begin{split}
		\|\widetilde{u}_n'\|_{L^2(\widetilde{\G})}^2=&\int_{\G/\mathcal{B}}|u_n'|^2\,dx+\frac{1}{2}\sum_{e\in\mathcal{B}}\int_{0}^{2|e|}|u_n'(x/2)|^2\,dx\\
		=&\int_{\G/\mathcal{B}}|u_n'|^2\,dx+\sum_{e\in\mathcal{B}}\int_e|u_n'|^2\,dx=\|u_n'\|_{L^2(\G)}^2		
		\end{split}
	\end{equation}
	
	\noindent Moreover, by construction, every value $t$ in the image of $\widetilde{u}_n$ is now attained at least twice on $\widetilde{\G}$. Therefore, by properties of the symmetric rearrangements and Gagliardo-Nirenberg inequality \eqref{EQ-critical GN} on $\rr$,
	
	\[
	\|\widetilde{u}_n\|_{L^6(\widetilde{\G})}^6=\|\widehat{(\widetilde{u}_n)}\|_{L^6(\rr)}^6\leq C_\rr\|\widehat{(\widetilde{u}_n)}\|_{L^2(\rr)}^4\|\widehat{(\widetilde{u}_n)}'\|_{L^2(\rr)}^2\leq C_\rr \|\widetilde{u}_n\|_{L^2(\widetilde{\G})}^4\|\widetilde{u}_n'\|_{L^2(\widetilde{\G})}^2
	\]
	
	\noindent that, combined with \eqref{EQ-u' doubled} and the fact that
	
	\[
	\begin{split}
	\|\widetilde{u}_n\|_{L^6(\widetilde{\G})}^6=&\|u_n\|_{L^6(\G)}^6+3\sum_{e\in\mathcal{B}}\int_e|u_n|^6\,dx\\
	\|\widetilde{u}_n\|_{L^2(\widetilde{\G})}^2=&\|u_n\|_{L^2(\G)}^2+3\sum_{e\in\mathcal{B}}\int_e|u_n|^2\,dx=\mu_\G+3\sum_{e\in\mathcal{B}}\int_e|u_n|^2\,dx
	\end{split}
	\]
	
	\noindent gives 
	
	\begin{equation}
	\label{EQ}
	\begin{split}
	\|u_n\|_{L^6(\G)}^6\leq&\|\widetilde{u}_n\|_{L^6(\widetilde{\G})}^6\leq C_\rr\Big(\mu_\G+3\sum_{e\in\mathcal{B}}\int_e|u_n|^2\,dx\Big)^2\|\widetilde{u}_n'\|_{L^2(\widetilde{\G})}^2\\
	\leq&C_\rr\Big(\mu_\G+3\sum_{e\in\mathcal{B}}\int_e|u_n|^2\,dx\Big)^2\|u_n'\|_{L^2(\G)}^2=\frac{3}{\mu_\rr^2}\Big(\mu_\G+3\sum_{e\in\mathcal{B}}\int_e|u_n|^2\,dx\Big)^2\|u_n'\|_{L^2(\G)}^2\,.
	\end{split}
	\end{equation}
	
	\noindent Now, since $u_n\to0$ in $L^\infty(\G)$ and $\mathcal{B}\subset\K_0$, then $\sum_{e\in\mathcal{B}}\int_e|u_n|^2\,dx\to0$ as $n\to+\infty$, and \eqref{EQ} implies, for n sufficiently large
	
	\[
	\frac{\|u_n\|_{L^6(\G)}^6}{\|u_n'\|_{L^2(\G)}^2}\leq3\frac{\mu_\G^2}{\mu_\rr^2}+o(1)
	\]
	
	\noindent that, coupled with \eqref{EQ-quotient to 3} and $\mu_\G<\mu_\rr$, provides a contradiction. Hence, $u\not\equiv0$.
	
	Repeating the same calculations we performed in the proof of Proposition \ref{PROP-compactness} allows to
	rule out also the case $\uLtwo^2<\mu$. Thus, $\uLtwo^2=\mu$, $u_n\to u$ strongly in $\HmuG$ and $u$	is a ground state.
	\end{proof}
		
	Notice that, as already highlighted in the Introduction, the assumption $\mu_\G<\mu_\rr$ in Theorem \ref{THM 3 INTRO} is essential, but we are not able to tell whether it is automatically satisfied by all graphs that do not fulfil (H$_{per}$) and without terminal edges. Nevertheless, we conclude this section with the following proposition, showing that the class of graphs for which $\mu_\G<\mu_\rr$ holds true is nonempty.
	
	\begin{prop}
		\label{prop-signpost} Let $\G$ be the periodic graph in Figure \ref{FIG-critical}(c). Then $\mu_\G<\mu_\rr$.
	\end{prop}

	\begin{proof}
		
		Let us first note that $\G$ as in Figure \ref{FIG-critical} is periodic according to Definition \ref{DEF-periodic} with periodicity cell $\K$ given by a vertical edge with a circle attached at one of its endpoints and a horizontal edge pointing to the right at the other. Denote then by $\Gamma,\,\mathcal{B},\,\mathcal{H}$ the circle, the vertical and the horizontal edge of $\K$ respectively, so that $\K=\Gamma\cup\mathcal{B}\cup\mathcal{H}$, and set $2\gamma:=|\Gamma|$, $2\beta:=|\mathcal{B}|$ and $|\mathcal{H}|:=\delta$. Moreover, let as usual $\Gamma_i,\mathcal{B}_i,\mathcal{H}_i$ be the corresponding parts of the $i$--th copy of $\K$ in $\G$, $\K_i$, for every $i\in\zz$.
		
		Since, by definition, $\mathcal{E}_\G(\mu_\G)=0$, in order to prove that $\mu_\G<\mu_\rr$, the idea is to exhibit explicitly a function $u\in H_{\mu_\rr}^1(\G)$ such that $E(u,\G)<0$. 
		
		To this aim, given $\lambda>0$, let $\phi_\lambda\in H_{\mu_\rr}^1(\rr)$ be the critical soliton \eqref{eq-crit sol} on the real line and consider $w_\lambda\in H^1(\G)$ defined by the following procedure (see Figure \ref{FIG-w}). 
		
		
		\begin{figure}
			\centering
			\subfloat[][]{
			\includegraphics[width=0.6\textwidth]{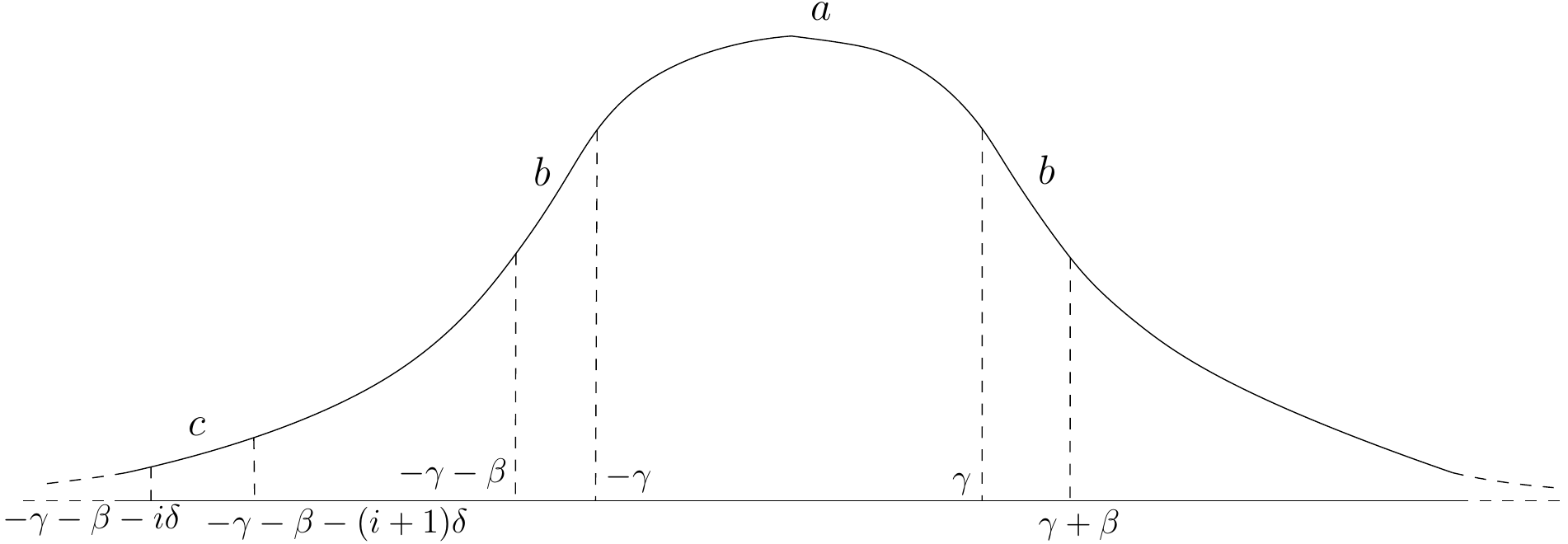}}\quad
			\subfloat[][]{
			\includegraphics[width=0.9\textwidth]{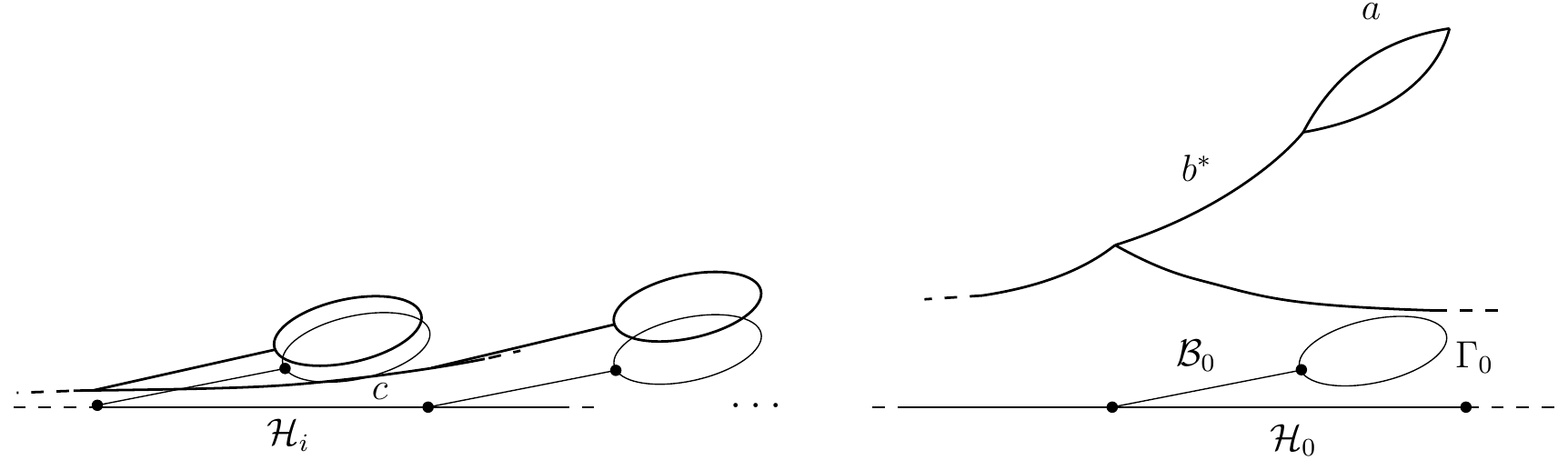}}
			\caption{(a) the soliton $\phi_\lambda$ on $\rr$ and (b) the corresponding function $w_\lambda$ on $\G$ as in the proof of Proposition \ref{prop-signpost}. The letters $a,\,b,\,c$ denote the restrictions of $\phi_\lambda$ to $(-\gamma,\gamma)$, $(-\gamma-\beta,-\gamma)\cup(\gamma,\gamma+\beta)$ and $(-\gamma-\beta-i\delta,-\gamma-\beta-(i+1)\delta)$ respectively (i<0), and $b^*$ denotes the decreasing rearrangement of $b$.}
			\label{FIG-w}
		\end{figure}
		
		First, let $\phi_{\lambda\mid_{(-\gamma,\gamma)}}$ be the restriction of $\phi_\lambda$ to the interval $(-\gamma,\gamma)$ and set
		
		\[
		w_{\lambda\mid_{\Gamma_0}}:=\phi_{\lambda\mid_{(-\gamma,\gamma)}}
		\]
		
		\noindent having identified $\Gamma_0$ with the interval $(-\gamma,\gamma)$ in such a way that both the endpoints $-\gamma,\,\gamma$ correspond to the vertex of $\Gamma_0$ attached to $\mathcal{B}_0$.
		
		Secondly, let $I:=[-\gamma-\beta,-\gamma]\cup[\gamma,\gamma+\beta)]$,  $\phi_{\lambda\mid_{I}}$ be the restriction of $\phi_\lambda$ to $I$ and $(\phi_{\lambda_I})^*$ be its decreasing rearrangement on $[0,2\beta]$. We then set
		
		\[
		w_{\lambda\mid_{\mathcal{B}_0}}:=(\phi_{\lambda\mid_I})^*
		\]
		
		\noindent provided the identification of $\mathcal{B}_0$ with the interval $[0,2\beta]$ so that the origin of $[0,2\beta]$ corresponds to the vertex that $\mathcal{B}_0$ shares with $\Gamma_0$. 
		
		Thirdly, for every $i\in\zz$, let either $I_i:=[\gamma+\beta+i\delta,\gamma+\beta+(i+1)\delta]$ if $i\geq0$, or $I_i:=[-\gamma-\beta-i\delta,-\gamma-\beta-(i+1)\delta]$ if $i<0$, and let $\phi_{\lambda\mid_{I_i}}$ be the restriction of $\phi_\lambda$ to $I_i$. Then, for every $i\in\zz$, we set
		
		\[
		w_{\lambda\mid_{\mathcal{H}_i}}:=\phi_{\lambda\mid_{I_i}}
		\]
		
		\noindent identifying $\mathcal{H}_i$ with $I_i$ so that the endpoint of $I_i$ with smallest absolute value corresponds to the vertex of $\mathcal{H}_i$ closest to $\K_0$.
		
		Finally, for every $i\in\zz\setminus\{0\}$, we set 
		
		\[
		w_{\lambda\mid_{\Gamma_i\cup\mathcal{B}_i}}:\equiv\begin{cases}
		\phi_\lambda(\gamma+\beta+i\delta) & \text{if }i>0\\
		\phi_\lambda(-\gamma-\beta-i\delta) & \text{if }i<0
		\end{cases}
		\]
		
		\noindent so that $w_\lambda$ is constant on every $\Gamma_i\cup\mathcal{B}_i$, $i\neq0$.
		
		Note that, by construction, $w_\lambda\in H^1(\G)$ and
		
		\[
		\begin{split}
		\|w_\lambda\|_{L^p(\Gamma_0)}&=\|\phi_\lambda\|_{L^p(-\gamma,\gamma)},\qquad\|w_\lambda\|_{L^p(\bigcup_{i\in\zz}\mathcal{H}_i)}=\|\phi_\lambda\|_{L^p(\rr\setminus(-\gamma-\beta,\gamma+\beta))}\qquad p\geq1\\
		\|w_\lambda'\|_{L^2(\Gamma_0)}&=\|\phi_\lambda'\|_{L^2(-\gamma,\gamma)},\qquad\|w_\lambda'\|_{L^2(\bigcup_{i\in\zz}\mathcal{H}_i)}=\|\phi_\lambda'\|_{L^2(\rr\setminus(-\gamma-\beta,\gamma+\beta))}\,,
		\end{split}
		\]
		
		\noindent whereas by the fact that the restriction of $\phi_\lambda$ to $(-\gamma-\beta,\gamma)\cup(\gamma,\gamma+\beta)$ attains exactly twice each value in its image and standard properties of decreasing rearrangements (see \cite[Lemma 2.1]{AST bound states})
		
		\[
		\begin{split}
		\|w_\lambda\|_{L^p(\mathcal{B}_0)}=&\|\phi_\lambda\|_{L^p((-\gamma-\beta,-\gamma)\cup(\gamma,\gamma+\beta))}\qquad p\geq1\\
		\|w_\lambda'\|_{L^2(\mathcal{B}_0)}\leq&\frac{1}{2}\|\phi_\lambda'\|_{L^2((-\gamma-\beta,-\gamma)\cup(\gamma,\gamma+\beta))}\,.
		\end{split}
		\]
		
		\noindent Hence, setting $c:=2\gamma+2\beta$, we have
		
		\[
		\begin{split}
		\int_\G|w_\lambda|^2\dx=&\int_\rr|\phi_\lambda|^2\dx+\int_{\bigcup_{i\in\zz\setminus\{0\}}(\Gamma_i\cup\mathcal{B}_i)}|w_\lambda|^2\dx=\mu_\rr+2c\sum_{i=1}^\infty|\phi_\lambda(c+i\delta)|^2\\
		\int_\G|w_\lambda|^6\dx=&\int_\rr|\phi_\lambda|^6\dx+\int_{\bigcup_{i\in\zz\setminus\{0\}}(\Gamma_i\cup\mathcal{B}_i)}|w_\lambda|^6\dx=\int_\rr|\phi_\lambda|^6\dx+2c\sum_{i=1}^\infty|\phi_\lambda(c+i\delta)|^6\\
		\end{split}
		\]
		
		\noindent and
		
		\[
		\begin{split}
		\int_\G|w_\lambda'|^2\dx=&\int_\rr|\phi_\lambda'|^2\dx-\int_{(-\gamma-\beta,-\gamma)\cup(\gamma,\gamma+\beta)}|\phi_\lambda'|^2\dx+\int_{\mathcal{B}_0}|w_\lambda'|^2\dx\\
		\leq&\int_\rr|\phi_\lambda'|^2\dx-\frac{3}{4}\int_{(-\gamma-\beta,-\gamma)\cup(\gamma,\gamma+\beta)}|\phi_\lambda'|^2\dx\,,
		\end{split}
		\]
		
		\noindent yielding at
		
		\begin{equation}
			\label{eq-E w}
			\begin{split}
			E(w_\lambda,\G)\leq &E(\phi_\lambda,\rr)-\frac{3}{8}\int_{(-\gamma-\beta,-\gamma)\cup(\gamma,\gamma+\beta)}|\phi_\lambda'|^2\dx-\frac{c}{3}\sum_{i=1}^\infty|\phi_\lambda(c+i\delta)|^6\\
			=&-\frac{3}{4}\int_{\gamma}^{\gamma+\beta}|\phi_\lambda'|^2\dx-\frac{c}{3}\sum_{i=1}^\infty|\phi_\lambda(c+i\delta)|^6
			\end{split}
		\end{equation}
		
		\noindent since $E(\phi_\lambda,\rr)=0$, for every $\lambda>0$.

		Now, recalling the explicit formula \eqref{eq-crit sol} for $\phi_\lambda$ leads to 
		
		\[
		\sum_{i=1}^\infty|\phi_\lambda(c+i\delta)|^2=\lambda\sum_{i=1}^\infty \text{sech}\Big(\frac{2}{\sqrt{3}}\lambda (c+i\delta)\Big)\sim\lambda\sum_{i=1}^\infty \frac{1}{e^{\frac{2}{\sqrt{3}}\lambda(c+i\delta)}}=\frac{\lambda}{e^{\frac{2}{\sqrt{3}}\lambda c}(e^{\frac{2}{\sqrt{3}}\lambda \delta}-1)}\sim\frac{\lambda}{e^{\frac{2}{\sqrt{3}}\lambda(c+\delta)}}
		\]
		
		\noindent as $\lambda\to+\infty$, so that we get
		
		\[
		\frac{\mu_\rr}{\|w_\lambda\|_{L^2(\G)}^2}\sim\frac{\mu_\rr}{\mu_\rr+\frac{\lambda}{e^{\frac{2}{\sqrt{3}}\lambda(c+\delta)}}}=1-r(\lambda)
		\]
		
		\noindent where 
		
		\[
		r(\lambda):=\frac{\lambda}{e^{\frac{2}{\sqrt{3}}\lambda(c+\delta)}\mu_\rr+\lambda}\sim\frac{\lambda}{e^{\frac{2}{\sqrt{3}}\lambda(c+\delta)}}\qquad\text{for }\lambda\to+\infty\,.
		\]
		
		\noindent Therefore, letting $u_\lambda:=\sqrt{\frac{\mu_\rr}{\|w_\lambda\|_{L^2(\G)}^2}}w_\lambda$, we have $u_\lambda\in H_{\mu_\rr}^1(\G)$ for every $\lambda$ and, provided $\lambda$ large enough,
		
		\[
		\begin{split}
		E(u_\lambda,\G)=&\frac{1}{2}\frac{\mu_\rr}{\|w_\lambda\|_{L^2(\G)}^2}\|w_\lambda'\|_{L^2(\G)}^2-\frac{1}{6}\Big(\frac{\mu_\rr}{\|w_\lambda\|_{L^2(\G)}^2}\Big)^3\|w_\lambda\|_{L^6(\G)}^6\\
		=&\frac{1}{2}(1-r(\lambda))\|w_\lambda'\|_{L^2(\G)}^2-\frac{1}{6}(1-r(\lambda))^3\|w_\lambda\|_{L^6(\G)}^6+o(1)\\
		\leq&\frac{1}{2}(1-r(\lambda))\|w_\lambda'\|_{L^2(\G)}^2-\frac{1}{6}(1-4r(\lambda))\|w_\lambda\|_{L^6(\G)}^6+o(1)\\
		=&(1-r(\lambda))E(w_\lambda,\G)+\frac{1}{2}r(\lambda)\|w_\lambda\|_{L^6(\G)}^6+o(1)
		\end{split}
		\]
		
		\noindent taking advantage of the fact that $r(\lambda)\to0$ as $\lambda\to+\infty$. Combining with \eqref{eq-E w}, we get the following estimate
		
		\begin{equation}
		\label{eq-E u}
		E(u_\lambda,\G)\leq (1-r(\lambda))\Big(-\frac{3}{4}\int_{\gamma}^{\gamma+\beta}|\phi_\lambda'|^2\dx-\frac{c}{3}\sum_{i=1}^\infty|\phi_\lambda(c+i\delta)|^6\Big)+\frac{1}{2}r(\lambda)\|w_\lambda\|_{L^6(\G)}^6+o(1)
		\end{equation}
		
		\noindent holding for $\lambda\to+\infty$.
		
		Clearly, all the terms involved in the above expression goes to $0$ as $\lambda$ becomes infinite. Note thus that multiplying the big bracket by $r(\lambda)$ makes the whole product go to $0$ faster than $r(\lambda)$. Hence, since the term $\frac{2}{3}r(\lambda)\|w_\lambda\|_{L^6(\G)}^6$ is asymptotic to $\lambda^2 r(\lambda)$ as $\lambda$ increases, to understand the asymptotic behaviour of the upper bound in \eqref{eq-E u} we can neglect $(1-r(\lambda))$ in the first addend, simply replacing it by $1$. 
		
		Furthermore, with computations similar to the ones we performed before, it is readily seen that
		
		\[
		\sum_{i=1}^\infty|\phi_\lambda(c+i\delta)|^6\sim\frac{\lambda^3}{e^{\frac{2}{\sqrt{3}}3\lambda (c+\delta)}}=o(r(\lambda))\qquad\text{as }\lambda\to+\infty
		\]
		
		\noindent so that we can restrict ourselves to focus only on the first integral in the big bracket in \eqref{eq-E u}.
		
		Since, differentiating \eqref{eq-crit sol},
		
		\[
		\phi_\lambda'(x)=-\frac{2}{\sqrt{3}}\lambda^{3/2}\text{sech}^{3/2}\Big(\frac{2}{\sqrt{3}}\lambda x\Big)\sinh\Big(\frac{2}{\sqrt{3}}\lambda x\Big)\,,
		\]
		
		\noindent it follows
		
		\[
		\int_{\gamma}^{\gamma+\beta}|\phi_\lambda'|^2\dx=\frac{\sqrt{3}}{2}\lambda^2\Big[\arctan(\tanh(x/2))-\frac{1}{2}\tanh x\, \text{sech}\, x\Big]_{\frac{2}{\sqrt{3}}\lambda\gamma}^{\frac{2}{\sqrt{3}}\lambda(\gamma+\beta)}
		\]
		
		\noindent and expanding to the first order, when $\lambda$ is large enough,
			
		\[
		\lambda^2\Big[\arctan(\tanh(x/2))-\frac{1}{2}\tanh x\, \text{sech}\, x\Big]_{\frac{2}{\sqrt{3}}\lambda\gamma}^{\frac{2}{\sqrt{3}}\lambda(\gamma+\beta)}\sim \frac{\lambda^2}{e^{\frac{2}{\sqrt{3}}\lambda\gamma}}\,,
		\]
		
		\noindent so that (recalling also $c=2\gamma+2\beta$)
		
		\[
		\frac{\int_{\gamma}^{\gamma+\beta}|\phi_\lambda'|^2\dx}{\lambda^2 r(\lambda)}\sim\frac{\frac{\lambda^2}{e^{\frac{2}{\sqrt{3}}\lambda\gamma}}}{\frac{\lambda^3}{e^{\frac{2}{\sqrt{3}}\lambda(c+\delta)}}}=\frac{ e^{\frac{2}{\sqrt{3}}\lambda(2\beta+\gamma+\delta)}}{\lambda}\to+\infty\qquad\text{for }\lambda\to+\infty\,.
		\]
		
		\noindent Summing up, plugging into \eqref{eq-E u} shows that
		
		\[
		E(u_\lambda,\G)\leq -\frac{3}{4}\int_\gamma^{\gamma+\beta}|\phi_\lambda'|^2\dx+\frac{2}{3}r(\lambda)\|w_\lambda\|_{L^6(\G)}^6+o(1)<0
		\]
		
		\noindent as soon as $\lambda$ is sufficiently large, and we conclude.
	\end{proof}

	\appendix
	\section{Appendix}
	\label{appendix}

	As pointed out in Section \ref{sec:periodic}, it is not immediate to see that assumptions \textit{(i)}-\textit{(ii)} on $D,R$ and $\sigma$ we require stating Definition \ref{DEF-periodic} are not somehow restrictive. Precisely, one may wonder whether dropping these assumptions allows to exhibit any graph which cannot be constructed if $D,R$ and $\sigma$ satisfy \textit{(i)}-\textit{(ii)}. The following two propositions provide a negative answer to this question, ensuring that no loss of generality is at stake.
	
	From now on, we will say that two graphs $\G = (V (\G),E(\G))$, $\G' = (V (\G'),E(\G'))$ are equal, and we will write $\G = \G'$, if there exist two bijections $\varphi: V (\G) \to V (\G')$,  $\psi:E(\G) \to E(\G')$ such that $e\in E(\G)$ is an edge between $v,w \in V (\G)$ if and only if  $\psi(e) \in E(\G')$ is an edge between $\varphi(v), \varphi(w) \in V (\G')$, and $\psi$ is measure-preserving, that is $|e|=|\psi(e)|$, for every $e\in E(\G)$.
		
	\begin{prop}
		\label{PROP APP 1}
		Let $\K$ be a fixed compact graph, $D,R$ two non-empty subsets of $V(\K)$ and $\sigma:D\to R$ bijective. Suppose $D\cap R\neq\emptyset$ and let $\G$ be as in Definition \ref{DEF-periodic}. Then, either
				
		\begin{itemize}
			\item[(a)] \text{diam}$(\G)<+\infty$, or
 			\item[(b)] there exists a compact graph $\K'$, two non-empty subsets $D',R'\subset V(\K')$ and a bijection $\sigma': D'\to R'$, with $D'\cap R'=\emptyset$, such that if $\G'$ is the periodic graph with periodicity cell $\K'$ and pasting rule $\sigma'$ as in Definition \ref{DEF-periodic}, then $\G=\G'$.
 		\end{itemize}
	\end{prop}

	\begin{proof}
		
		Let $D\cap R=\{x^1,\,.\,.\,.\,,\,x^n\}$, for some $x^1,\,.\,.\,.\,,\,x^n\in\K$.
		Moreover, for every $i\in\zz$ and every $j\in\{1,\,.\,.\,.\,,\,n\}$, denote by $x_i^j\in\G$ the copy of $x^j$ belonging to $\K_i$.
		
		We split the proof into two parts.

		\medskip
		\textit{Part (i).} Suppose that there exists a subset $\{x^{j_1},\,.\,.\,.\,,\,x^{j_l}\}$ of $D\cap R$ such that (see Figure \ref{FIG-star})
		
		 \[\sigma(x^{j_1})=x^{j_2},\,.\,.\,.\,,\,\sigma(x^{j_{l-1}})=x^{j_l},\,\sigma(x^{j_l})=x^{j_1}\,.
		 \]
		
		 \noindent Note that, building up $\G$ according to the pasting rule $\sigma$ as in Definition \ref{DEF-periodic}, for every $i_1,i_2\in\zz$, $s,t\in\{j_1,\,.\,.\,.\,,\,j_l\}$, $x_{i_1}^{s}$ and $x_{i_2}^{t}$ correspond to the same point of $\G$ if and only if $|s-t|=|i_1-i_2|$ mod $l$.
		
		Consider now $y,z\in\G$ and let $i_1,i_2\in\zz$ be such that $y\in\K_{i_1},\,z\in\K_{i_2}$. Moreover, let $s,t\in\{j_1,\,.\,.\,.\,,\,j_l\}$ be such that $|s-t|=|i_1-i_2|$ mod $l$. We get
		
		\[
		d(y,z)\leq d(y,x_{i_1}^{s})+d(x_{i_2}^{t},z)\leq 2\text{diam}(\K)
		\]
		
		\noindent and passing to the supremum over all $y,z\in\G$
		
		\[
		\text{diam}(\G)\leq2\text{diam}(\K)<+\infty
		\]
		
		\noindent so that case \textit{(a)} occurs.
		
		
		\begin{figure}[t]
			\centering
			\subfloat[][]{
				\includegraphics[width=0.25\columnwidth]{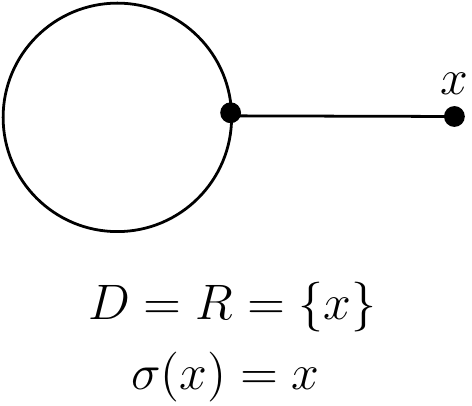}}\qquad\qquad
			\subfloat[][]{
				\includegraphics[width=0.25\columnwidth]{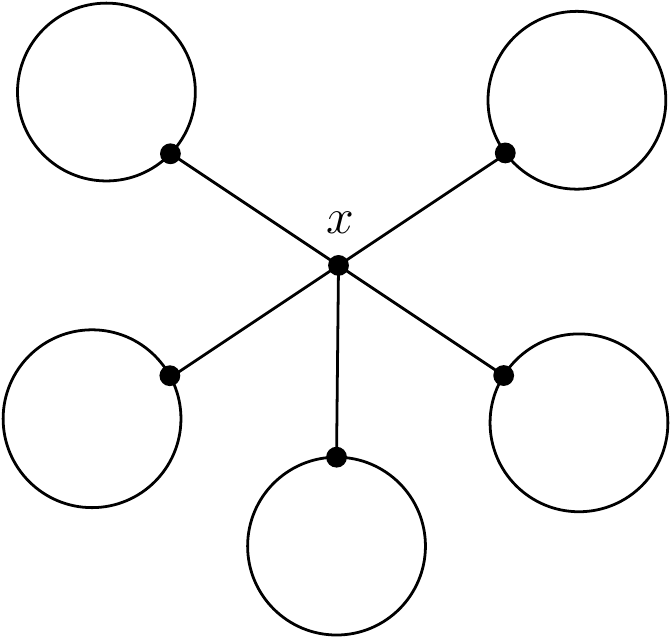}}
			\caption{example of $D,\,R$ and $\sigma$ as in part (i) of the proof of Proposition \ref{PROP APP 1} and the resulting graph $\G$.}
			\label{FIG-star}
		\end{figure}
		
		
		\medskip
		\textit{Part (ii).} Suppose now that no subset of $D\cap R$ satisfies the assumption of part (i). 
		
		Let us begin to deal with the case $\sigma(x^i)\neq x^j$, for every $i,j\in\{1,\,.\,.\,.\,,\,n\}$. Hence, for every $j\in\{1,\,.\,.\,.\,,\,n\}$, there exist $y^j\in D,\,z^j\in R$ be such that
		
		\[
		\begin{split}
			\sigma(y^j)=&x^j\\
			\sigma(x^j)=&z^j
		\end{split}
		\]
		\noindent and both $y^j\neq x^i$ and $z^j\neq x^i$, for every $i,j$.
				
		We then introduce the following construction. Consider two copies of $\K$, namely $\K_1,\K_2$, and paste them together according to $\sigma$, that is, looking at $\sigma$ as a map from $D_1$ into $R_2$, identify each element of $D_1$ with its image through $\sigma$ in $R_2$. This defines a compact graph $\K'=(V(\K'),E(\K'))$ with
		
		\[
		\begin{split}
		V(\K')=&\Big(V(\K_1)\cup V(\K_2)\Big)_{/\{a\sim b\,\Longleftrightarrow\,a\in D_1,\,b\in R_2,\,\sigma(a)=b\}}\\
		E(\K')=&E(\K_1)\cup E(\K_2)\,.
		\end{split}
		\]
		
		\noindent Notice that, as vertices of $\K'$, each $y_1^j$ is identified with the corresponding $x_2^j$, and each $x_1^j$ is identified with the corresponding $z_2^j$, for every $j\in\{1,\,.\,.\,.\,,\,n\}$. We denote by $v^{1,j}\in V(\K')$ the element of $V(\K')$ resulting from the identification of $y_1^j$ and $x_2^j$, and by $v^{2,j}\in V(\K')$ the element of $V(\K')$ resulting from the identification of $x_1^j$ and $z_2^j$. 
		
		On the one hand, since $\sigma(x^i)\neq x^j$, for every $i,j$, no $x_1^i$ is identified with any $x_2^j$, yielding at $v^{1,j}\neq v^{2,j}$, for every $i,j\in\{1,\,.\,.\,.\,,\,n\}$.
		
		On the other hand, as $x_2^j\in D_2$ for every $j$, thinking of $D_2$ as a subset of $V(\K')$, we have $v^{1,j}\in D_2$. Similarly, as $x_1^j\in R_1$ for every $j$, thinking of $R_1$ as a subset of $V(\K')$, we get $v^{2,j}\in R_1$.

		Therefore, setting $D', R'\subset V(\K')$ to be
		
		\[
		D':=D_2\qquad R':=R_1
		\]
		
		\noindent and noting that $D_2=(D_2/(D_2\cap R_2))\cup\{v^{1,j}\,:\,j=1,\,.\,.\,.\,,\,n\}$ and $R_1=(R_1/(D_1\cap R_1))\cup \{v^{2,j}\,:\,j=1,\,.\,.\,.\,,\,n\}$, it follows that $D'\cap R'=\emptyset$.

		We then define $\sigma': D'\to R'$
		
		\[
		\sigma'(v):=\begin{cases}
		\sigma(v) & \text{if }v\neq v^{1,j}, y_2^j,\,\text{for every }j=1,\,.\,.\,.\,,\,n\\
		v^{2,j}   & \text{if }v=y_2^j,\,\text{for some }j=1,\,.\,.\,.\,,\,n\\
		z_1^j & \text{if }v=v^{1,j},\,\text{for some }j=1,\,.\,.\,.\,,n\,,
		\end{cases}
		\]
		
		\noindent which is a bijection by construction and the fact that $\sigma$ is a bijection. 
		
		It is readily seen that pasting together $n$ copies $\K_1',\,.\,.\,.\,,\,\K_n'$ of $\K'$ according to $\sigma'$, identifying each element of $D_i'$ with its image through $\sigma'$ in $R_{i+1}'$, for every $i=1,\,.\,.\,.\,,\,n-1$, is equivalent to paste together $2n$ copies $\K_1,\,.\,.\,.\,,\,\K_{2n}$ of $\K$ according to $\sigma$, identifying each element of $D_i$ with its image through $\sigma$ in $R_{i+1}$, for every $i=1,\,.\,.\,.\,,\,2n-1$.

		Hence, let $\G'$ be the periodic graph determined by the periodicity cell $\K'$ and the pasting rule $\sigma'$ according to Definition \ref{DEF-periodic}. Considering, for every $i\in\zz$, the natural bijections between the vertices and edges of $\K_i'$ and the ones of $\K_{2i-1}\cup\K_{2i}$, it follows that $\G'=\G$, and case \textit{(b)} holds.

		
		\begin{figure}[t]
			\centering
			\subfloat[][]{
				\includegraphics[width=0.25\textwidth]{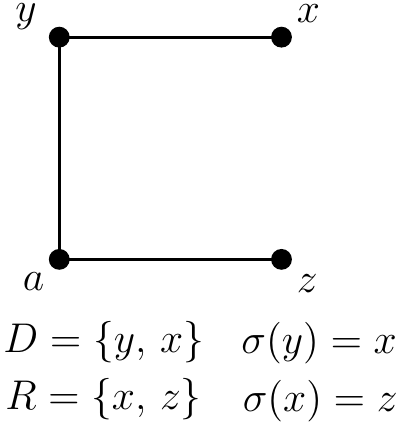}
			}\qquad\qquad
			\subfloat[][]{
				\includegraphics[width=0.25\textwidth]{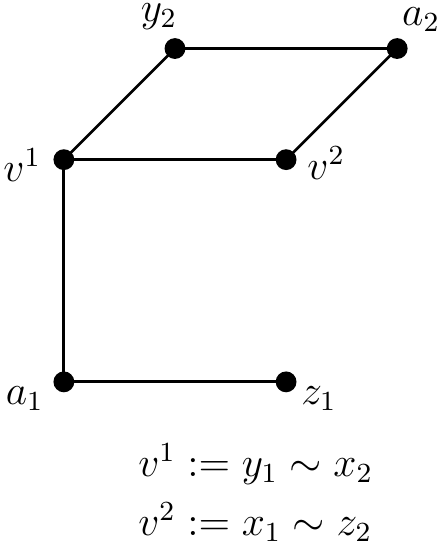}
			}
			
			\subfloat[][]{
				\includegraphics[width=0.7\textwidth]{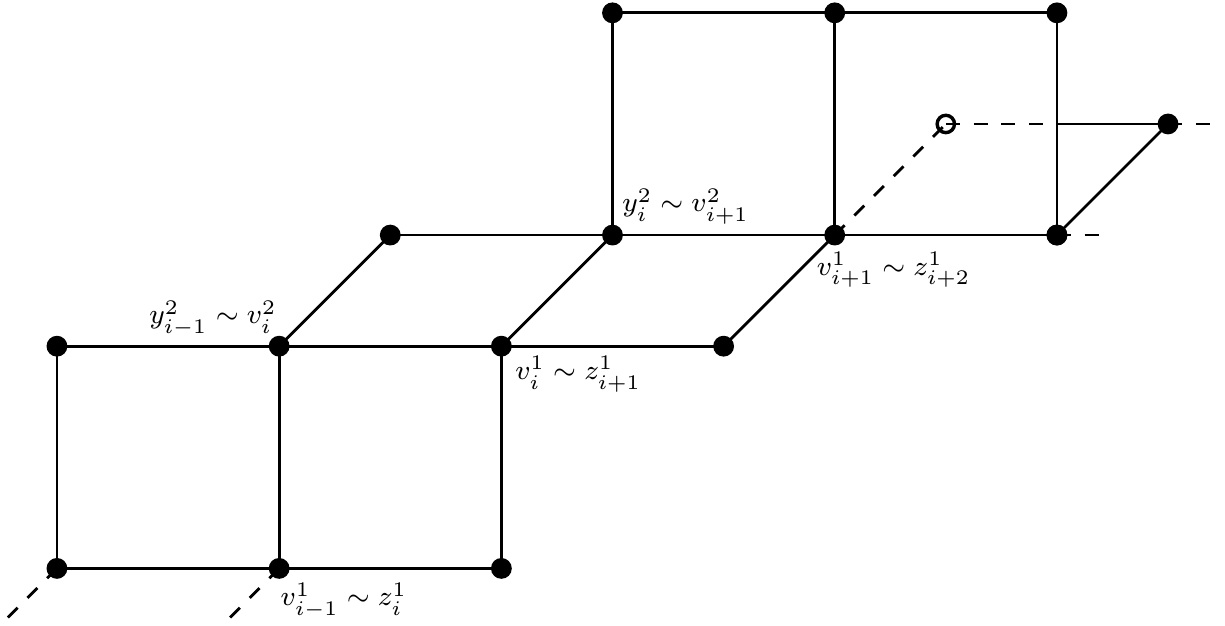}
			}
			\caption{example of $D,\,R$ and $\sigma$ as in part (ii) of the proof of Proposition \ref{PROP APP 1}, the graph $\K'$, and the resulting graph $\G$.}
			\label{FIG-part ii}
		\end{figure}
		

		To conclude, it remains to drop the assumption that $\sigma(x^i)\neq x^j$, for every $i,j\in\{1,\,.\,.\,.\,,n\}$, which only reflects into a minor modification of the previous argument. 
		
		Indeed, suppose that there exist $x^{j_1},\,.\,.\,.\,,\,x^{j_m}$ such that $\sigma(x^{j_1})=x^{j_2}$, $\sigma(x^{j_2})=x^{j_3},\,.\,.\,.\,,\,\sigma(x^{j_{m-1}})=x^{j_m}$. Moreover, let $y^{j_1}, z^{j_m}$ be so that $\sigma(y^{j_1})=x^{j_1}$, $\sigma(x^{j_m})=z^{j_m}$, and $y^{j_1},z^{j_m}\neq x^i$, for every $i=1,\,.\,.\,.\,,\,n$.
		
		For the sake of simplicity, assume also that no other $i,j$ satisfies $\sigma(x^i)=x^j$. However, this does not cause any loss of generality, and what follows straightforwardly adapts to cover the more general situation.
		
		We then consider $m+2$ copies of $\K$, say $\K_0,\,\K_1,\,.\,.\,.\,,\,\K_{m+1}$, and we paste together $\K_i$ with $\K_{i+1}$, for every $i=0,\,.\,.\,.\,,\,m$, according to $\sigma$, that is, we identify each element of $D_i$ with its image through $\sigma$ in $R_{i+1}$, for every $i=0,\,.\,.\,.\,,\,m$. A compact graph $\K'$ arises from this procedure, and we introduce a shorthand notation for the following identifications
		
		\[
		\begin{split}
		v^{1,1}:=&x_{m+1}^{j_m}\sim x_{m}^{j_{m-1}}\sim\,.\,.\,.\,\sim x_1^{j_1}\sim y_0^{j_1}\\
		v^{1,2}:=&x_{m+1}^{j_{m-1}}\sim x_m^{j_{m-2}}\sim\,.\,.\,.\,\sim x_2^{j_1}\sim y_1^{j_1}\\
		.\,.\,.&\\
		v^{1,m}:=&x_{m+1}^{j_1}\sim y_m^{j_1}
		\end{split}
		\]
		
		\noindent and
		
		\[
		\begin{split}
		v^{2,1}:=&x_0^{j_1}\sim x_1^{j_2}\sim\,.\,.\,.\,\sim x_{m-1}^{j_m}\sim z_m^{j_m}\\
		v^{2,2}:=&x_0^{j_2}\sim x_1^{j_3}\sim\,.\,.\,.\,\sim x_{m-2}^{j_m}\sim z_{m-1}^{j_m}\\
		.\,.\,.&\\
		v^{2,m}:=&x_0^{j_m}\sim z_1^{j_m}\,.
		\end{split}
		\]

		\noindent As it is immediate to verify, $v^{1,i}\neq v^{2,j}$, for every $i,j\in\{1,\,.\,.\,.\,,\,m\}$. Furthermore, looking at $R_0$ and $D_{m+1}$ as subsets of $V(\K')$, we get  $v^{1,i}\in D_{m+1}$ and $v^{2,i}\in R_0$ for every $i=1,\,.\,.\,.,\,m$.
		
		Therefore, setting
		
		\[
		D':=D_{m+1}\qquad R':=R_0
		\]
		
		\noindent and $\sigma':D'\to R'$
		
		\[
		\sigma'(v):=\begin{cases}
		\sigma(v) & \text{if }v\neq v^{1,j}\,\text{for every }j=1,\,.\,.\,.\,,\,m\\
		v^{2,m+2-j} & \text{if }v=v^{1,j},\,\text{for some }j=2,\,.\,.\,.\,,\,m\\
		z_0^{j_m} & \text{if }v=v^{1,1}\,,
		\end{cases}
		\]
		
		\noindent it can be easily verified that $D'\cap R'=\emptyset$, $\sigma'$ is a bijection from $D'$ to $R'$ and that the periodic graph $\G'$ given by Definition \ref{DEF-periodic} with periodicity cell $\K'$ and pasting rule $\sigma'$ satisfies $\G'=\G$, so that \textit{(b)} is proved to hold again.
\end{proof}

	\begin{rem}
		Metric graphs as in case \textit{(a)} of Proposition \ref{PROP APP 1} can be called \textit{star-like graphs} (see Figure \ref{FIG-star}). Since graphs like this do not satisfy $\text{diam}(\G)=+\infty$, we do not want to take them into account in the present paper, so that assuming $D\cap R=\emptyset$ in Definition \ref{DEF-periodic} is actually useful to avoid such situation. 
	\end{rem}

	\begin{prop}
		\label{PROP APP 2}
		Let $\K$ be a fixed compact graph, $D,R$ two non-empty subsets of $V(\K)$ such that $D\cap R=\emptyset$, and $\sigma:D\to R$. Suppose $\sigma$ is not bijective and let $\G$ be as in Definition \ref{DEF-periodic}. Then, there exists a compact graph $\K'$, two non-empty subsets $D',R'\subset V(\K')$, with $D'\cap R'=\emptyset$, and a bijection $\sigma':D'\to R'$, such that if $\G'$ is the periodic graph with periodicity cell $\K'$ and pasting rule $\sigma'$, then $\G=\G'$.
	\end{prop}

	\begin{proof}
		
		Note first that we can always assume that $\sigma$ is surjective. Indeed, if this is not the case, we simply redefine $R$ getting rid of the vertices with no pre-images in $D$.
		
		Suppose now that $\sigma$ is not injective. Without loss of generality, assume that there exist only two vertices $s,t\in D$ so that $\sigma(s)=\sigma(t)=\overline{r}$, for some $\overline{r}\in R$, as the argument we discuss below plainly generalizes to the case of more than one couple of elements of $D$ sharing the same image through $\sigma$.
		
		Let us now introduce the following construction (see Figure \ref{FIG-app 2}). Starting from $\K$, we identify $s$ and $t$, thus defining a new graph $\K'=(V(\K'), E(\K'))$ so that
		
		\[
		\begin{split}
		V(\K')=&V(\K)_{/\{s\sim t\}}\\
		E(\K')=&E(\K)\,.
		\end{split}
		\]
		
		\noindent By construction, as vertices of $\K'$, $s$ and $t$ correspond to the same element, say $\overline{v}\in V(\K')$. Moreover, as $s,t\in D$, thinking of $D$ as a subset of $V(\K')$, we get $\overline{v}\in D$. Hence, setting
		
		\[
		D':= D\qquad R':=R
		\]
		
		\noindent and defining $\sigma':D'\to R'$ as
		
		\[
		\sigma'(v):=\begin{cases}
		\sigma(v) & \text{if }v\neq\overline{v}\\
		\overline{r} & \text{if }v=\overline{v}\,,
		\end{cases}
		\]
		
		\noindent it is immediate to see that $D'\cap R'=\emptyset$ and $\sigma'$ is a bijection from $D'$ into $R'$. Moreover, notice that pasting together two copies of $\K'$ according to $\sigma'$ is equivalent to paste together two copies of $\K$ according to $\sigma$.
		
		Therefore, let $\G'$ be the periodic graph as in Definition \ref{DEF-periodic} with periodicity cell $\K'$ and pasting rule $\sigma'$. Considering, for every $i\in\zz$, the natural bijections between the vertices and the edges of $\K_i'$ and the ones of $\K_i$, we have $\G=\G'$ and we conclude.
	\end{proof}
	
	
	\begin{figure}[t]
		\centering
		\subfloat[][]{
		\includegraphics[width=0.2\textwidth]{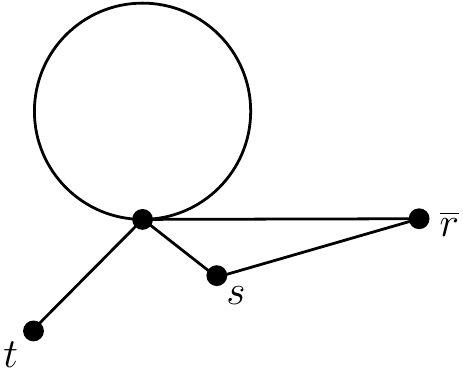}
		}\quad
		\subfloat[][]{
		\includegraphics[width=0.2\textwidth]{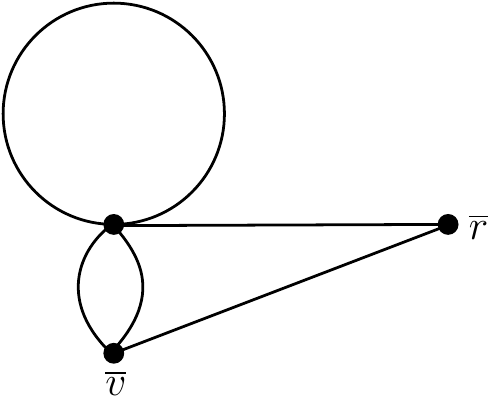}}\quad
		\subfloat[][]{
		\includegraphics[width=0.4\textwidth]{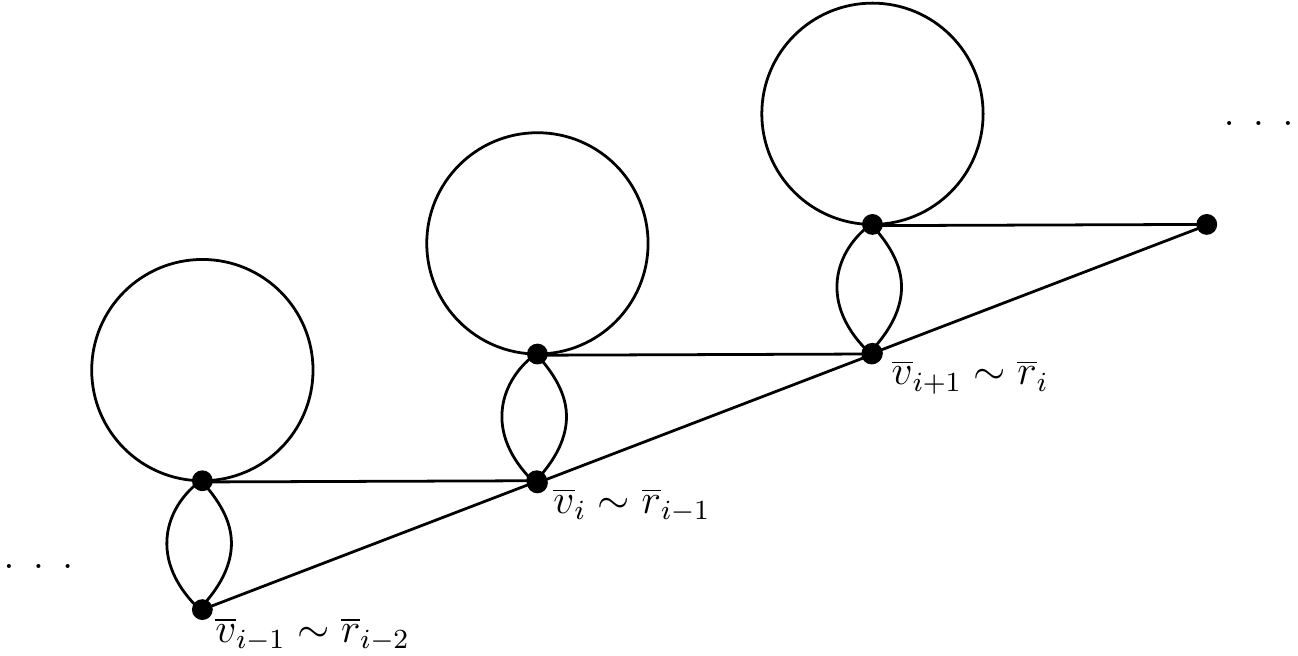}}
		\caption{example of the construction in the proof of Proposition \ref{PROP APP 2}.}
		\label{FIG-app 2}
	\end{figure}
	
	\textbf{Acknowledgments}
	
	\noindent The author is grateful to Riccardo Adami and Enrico Serra for fruitful discussions during the preparation of this work.

\end{document}